\documentclass{elsarticle}

\usepackage{lineno,hyperref}
\modulolinenumbers[5]

\journal{Discrete Applied Mathematics}

\usepackage{amsmath}
\usepackage{amssymb}
\usepackage{amsthm} 
\usepackage{colortbl}

\newcommand{\s}{{\bf ms}} 
\newcommand{\cs}{{\bf cms}}
\newcommand{\ms}{{\bf mms}} 
\newcommand{\cms}{{\bf cmms}}
\newcommand{\rep}{{\bf enh}}
\newcommand{\rrev}{{\bf rev}}
\newcommand{\cmp}{{\bf cmp}}
\newcommand{\obs}{{\bf obs}}

\newtheorem{observation}{Observation}
\newtheorem{corollary}{Corollary}
\newtheorem{lemma}{Lemma}

\begin{document}

\begin{frontmatter}

\title{Contraction Obstructions for Connected\\ Graph Searching\tnoteref{t1}}
\tnotetext[t1]{The second author of this paper where funded by the Natural Science and 
Engineering Research Council of Canada, Mitacs Inc, and the University of British Columbia. 
The third and the fourth author of this paper where co-financed by the E.U. 
(European Social Fund - ESF) and Greek national funds through the Operational Program 
``Education and Lifelong Learning'' of the National Strategic Reference Framework 
(NSRF) - Research Funding Program: ``Thales. Investing in knowledge society through 
the European Social Fund''.}
\author[BC]{Micah J Best}
\ead{mjbest@cs.ubc.ca}
\author[BC,MITACS]{Arvind Gupta}
\ead{arvind@mitacs.ca}
\author[UOA,AlGCo]{Dimitrios M. Thilikos}
\ead{sedthilk@thilikos.info}
\author[UOA]{Dimitris Zoros\corref{cor1}}
\ead{dzoros@math.uoa.gr}

\cortext[cor1]{Corresponding author}
\address[BC]{Department of Computer Science, University of British Columbia, B.C., Canada.}
\address[MITACS]{Mathematics of Information Technology \& Complex Systems (MITACS).}
\address[UOA]{Department of Mathematics, National and Kapodistrian University of Athens, Greece.}
\address[AlGCo]{AlGCo project team, CNRS, LIRMM, France.}

\begin{abstract}
\noindent We consider the connected variant of the classic mixed search game where, in each search 
step, cleaned edges form a connected subgraph. We consider graph classes with bounded connected 
(and monotone) mixed  search number and we deal with the question whether the obstruction set, 
with respect of the contraction partial ordering, for those classes is finite. In general, there is no 
guarantee that those sets are finite, as graphs are not well quasi ordered under the contraction partial 
ordering relation.
 In this paper we provide the obstruction set for $k=2$, where $k$ is the number of searchers we are 
allowed to use. This set is finite, it consists of 177 graphs and completely characterises the graphs 
with connected (and monotone) mixed  search number at most 2. Our proof reveals that the ``sense 
of direction" of an optimal search searching is important for connected search which is in contrast to 
the unconnected original case. We also give a double exponential lower bound on the size of the 
obstruction set for the classes where this set is finite.
\end{abstract}


\end{frontmatter}


\section{Introduction}

A {\em mixed searching game} is defined in terms of a
 graph representing a system of tunnels where an agile and omniscient fugitive with 
unbounded speed is hidden (alternatively, we can formulate the same problem 
considering that the tunnels are contaminated by some poisonous gas). 
The fugitive is occupying the edges of the graph and the searchers can be placed 
on its vertices. In the beginning of the game, the fugitive chooses some edge
and there are no searchers at all on the graph.
The objective of the searchers is to deploy a search strategy on the graph that 
can guarantee the capture of the fugitive. The fugitive is {\em captured}  if at some 
point he resides on an edge $e$ and one of the following capturing cases occurs.

\begin{itemize}
\item[{\bf A}:] {\em both  endpoints of $e$ are occupied by a searcher,}
\item[{\bf B}:] {\em a searcher slides along $e$}, i.e., a searcher is moved from one
endpoint of the edge to the other endpoint. 
\end{itemize}
 
\noindent A {\em search strategy} on a graph $G$ is a finite sequence ${\cal S}$ containing  
moves of the following types.

\begin{itemize}
\item[]\hskip-0.5cm{\sf p}$(v)$: placing a new searcher on a vertex $v$, 
\item[]\hskip-0.5cm{\sf r}$(v)$: deleting a searcher 
from a vertex $v$, 
\item[]\hskip-0.5cm{\sf s}$(v,u)$: sliding a searcher on $v$ along the edge $\{v,u\}$ and placing it on $u$. 
\end{itemize}
We stress that the fugitive is {\sl agile} and {\sl omniscient}, i.e. he moves at any time 
in the most favourable, for him, position and is {\sl invisible}, i.e. the searchers strategy 
is given ``in advance'' and does not depend on the moves of the fugitive during it.\medskip

Given a search ${\cal S}$, we denote by $E({\cal S},i)$ the set of edges that are clean 
after applying the first $i$ steps of ${\cal S}$, where by ``clean'' we mean that the search strategy
can guarantee that none of its edges will be occupied by the fugitive after the $i$-th step.
More formally,  we set $E({\cal 
S},0)=\emptyset$ and in step $i>0$ we define  $E({\cal S},i)$
as the set defined as follows: first consider the set $Q_{i}$ containing  
all the edges in $E({\cal S},i-1)$ plus the edges of
$E^{(i)}$ the set of edges that are  cleaned after the $i$-th move because of the application of  cases ${\bf A}$ or ${\bf B}$.
Notice that $E^{(i)}$ may be empty. In particular, it may be non-empty  in 
case the $i$-move is a placement move, will always be  empty
in case the $i$-th move is a removal move and  will surely be non-empty in case the $i$-th move
is a sliding move. In the third case, the edge along which the sliding occurs 
is called {\em the sliding edge} of $E^{(i)}$.
Then, the set $E({\cal S},i)$ is defined as the set of all edges in $Q_{i}$
minus those for which there is a path starting from them 
and finishing in an edge not in $Q_{i}$. This expresses the fact that the agile and omniscient fugitive 
could use any of these paths in order to occupy again some of the edges 
in $Q_{i}$.
In case $E({\cal S},i)\subset Q_{i}$, we say that the $i$-th move is a {\em recontamination move}. Notice that in such a 
case we have that $E({\cal S},i-1)\not \subseteq E({\cal S},i)$.

The object of a mixed search is to clear all edges using a search. We
call search ${\cal S}$ {\em complete} if at some step all  edges of $G$ are clean, i.e.
$E({\cal S},i)=E(G)$ for some $i$.
 
 \paragraph{Connected monotone mixed search number}
 
The mixed search number of a search is the maximum number of searchers on the graph during 
any move. 
 A search without recontamination moves is 
called {\em monotone}. Mixed search number has been introduced in~\cite{BienstockS91mono}. 
The mixed  search number, $\s(G)$, of a 
graph $G$ is the minimum mixed search number over all the possible complete searches on it 
(if $G$ is an edgeless graph, then 
this number is $0$). 
A search is {\em connected} if $E({\cal S},i)$ induces a 
connected subgraph of $G$ for every step $i$. Given a graph $G$, we will denote the minimum 
mixed search number over all the possible complete connected searches on it by $\cs(G)$ and 
we will call this number connected mixed  search number of 
$G$. The monotone (resp. connected monotone) mixed search number, $\ms(G)$ (resp. 
$\cms(G)$), of $G$ is the minimum mixed search number over all the possible complete 
monotone (connected monotone) searches of it (connected variants are defined only under 
the assumption that $G$ 
is a connected graph). The concept of connectivity in graph searching was introduced for the 
first time  in~\cite{BarriereFFFNST12conn}
and was motivated by application of graph searching where the ``clean" territories  should be 
maintained connected so to guarantee the safe communication between the searchers during 
the search.

\paragraph{Obstructions}

Given a graph invariant ${\bf p}$, a partial ordering relation on graphs $\unlhd$, and an integer 
$k$ we denote by 
$\obs_{\unlhd}({\cal G}[{\bf p},k])$ the set of all $\unlhd$-minimal graphs $G$ where ${\bf p}(G)>k$
and we call 
it {\em the $k$-th $\unlhd$-obstruction set for ${\bf p}$}. We also say that {\em ${\bf p}$ is closed 
under $\unlhd$} if for every two graphs $H$ and $G$, $H\unlhd G$ implies 
that ${\bf p}(H)\unlhd {\bf p}(G)$.
Clearly, if ${\bf p}$ is closed under $\unlhd$, then the  $k$-th $\unlhd$-obstruction set for ${\bf p}$
provides a complete characterisation for the class ${\cal G}_k=\{G \mid {\bf p}(G)\unlhd k\}$:
a graph belongs in ${\cal G}_{k}$ iff none of the graphs in the $k$-th $\unlhd$-obstruction set for 
${\bf p}$ is contained in $G$ with respect to the relation $\unlhd$.

\paragraph{Our results}

In this paper we are interested in obstruction characterisations for the graphs with bounded 
connected (monotone) mixed search number.
While it is known that ${\bf ms}$ is closed under taking of minors, this is not he case for $\cs$ 
and ${\bf cmms}$ 
where the connectivity requirement applies.  From Robertson and Seymour 
Theorem~\cite{RobertsonS-XX}, the  $k$-th $\leq$-obstruction 
set for ${\bf ms}$ is always finite, {where $\leq$ is the minor partial ordering relation (defined 
formally in Subsection~\ref{cont8e})}. Moreover this set 
has been found for $k=1$ (2 graphs) and $k=2$ (36 graphs) in~\cite{TakahashiUK95b}. 
However, no such result exists for the  obstruction characterisations of the connected monotone 
mixed search number.
As we prove in this paper, $\cs$ and ${\bf cmms}$ are closed under contractions. 
Unfortunately, graphs are not well quasi ordered with respect to the contraction relation, 
therefore there is no guarantee that 
the $k$-th contraction obstruction set for $\cs$ or ${\bf cmms}$ is finite for all $k$. The finiteness 
of this set 
is straightforward if $k=1$ as $\obs_{\preceq}({\cal G}[\cms,1])=\{K_{3},K_{1,3}\}$. 
In this paper we  completely resolve the case where $k=2$. We prove that 
$\obs_{\preceq}({\cal G}[\cms,2])=\obs_{\preceq}({\cal G}[\cs,2])$ and we prove that this set  is 
finite by providing
all 177 graphs that it contains. The proof of our results is based on a series of lemmata that 
confine the structure of the 
graphs with connected monotone mixed search number at most 2. We should stress that, 
in contrary to the case of ${\bf ms}$ 
the direction of searching is crucial for ${\bf cmms}$. This makes the detection of the 
corresponding obstruction sets 
more elaborated as special obstructions are required in order to force a certain sense of direction 
in the search strategy.
For this reason, our proof makes use of a more  general 
variant of the mixed search strategy that forces the searchers to start and finish to specific sets of 
vertices. Obstructions for this more general type of searching are combined in order to form the 
required obstructions for ${\bf cmms}$. We also give a double exponential lower bound on the size of 
the contraction obstruction set for the classes with bounded connected monotone search number. 
This lower bound is only meaningful for the classes where this obstruction set is finite. 

\section{Preliminary Definitions and Results} 
\label{sec:definitions}

Let $A$ be a set and let ${\cal A}=\langle a_{1},\ldots,a_{r}\rangle$
be an ordering of $A$. We denote by ${\bf prefsec}({\cal A})$
the ordering $\langle A_{0},\ldots,A_{r}\rangle$ of  subsets of $A$, where 
$A_{0}=\emptyset$ and for $i=1,\ldots,r$,   $A_{i}=\{a_{1},\ldots,a_{i}\}$. Let ${\cal A}_1$ and 
${\cal A}_2$ 
be two disjoint orderings of  $A$, we denote by ${\cal A}_1\oplus {\cal A}_2$ 
the concatenation of these two orderings.

All graphs under consideration will be finite, without loops or multiple edges.
Let $G$ be a graph and  $e=\{u,v\}\in E(G)$ an edge. The {\em edge-contraction} of $\{u,v\}$ 
(or just the {\em contraction} of $\{u,v\}$) is the operation  that deletes this edge, adds a new 
vertex $x_{uv}$ and connects this vertex to all the neighbours of $u$ and $v$ (if some multiple 
edges are created we delete them). We denote by $G/e$ the graph
obtained from $G$ by contracting edge $e$.

If $S\subseteq V(G)$ we call graph $G[S]=(S,\big\{\{u,v\}\in E(G)\mid u,v\in
S\big\})$ the \textit{subgraph of $G$ induced by $S$}. Also, given a set $F\subseteq E(G)$
 we call graph $G[F]=(\bigcup_{e\in F}e,F)$ the \textit{subgraph of $G$ induced by
$F$} and  we denote by $V(F)$ the set of vertices in $G[F]$.

We define the {\em union} of two graphs $G_1= (V_1 , E_1)$ and $G_2 = (V_2 , E_2)$ to be 
the graph $G_1\cup G_2=(V_1\cup V_2, E_1\cup E_2)$. When $V_1$ and $V_2$ are disjoint, 
we refer to this union as the {\em disjoint union} of $G_1$ and  $G_2$ (we denote it as 
$G_{1}+G_{2}$).

A vertex of a graph is called {\em pendant} if it has degree at most $1$.
An edge $e$ of a graph $G$ is \textit{pendant} if  one of its endpoints is pendant.
If both endpoints of an edge of $G$ are pendant, then we say that $e$ is an isolated edge.

We adapt the standard notations for the neighbourhood and the degree of a vertex $u\in V(G)$, i.e. 
the set off all vertices connected with $u$ by an edge and the cardinality of this set, which is 
$N_{G}(u)$ 
and $\deg_{G}(u)$ respectively.

\subsection{Rooted graph triples.}

A {\em rooted graph triple}, or, for simplicity, a {\em rooted graph}, is an ordered  triple 
$(G,S^{\rm in},S^{\rm out})$
where $G$ is a connected graph and  $S^{\rm in}$ and $S^{\rm out}$ are subsets of $V(G)$ 
($S^{\rm in}$ and $S^{\rm out}$
are not necessarily disjoint sets). If ${\bf G}=(G,S^{\rm in},S^{\rm out})$ then we also say that 
$\bf G$ is the 
graph $G$ {\em in-rooted} on $S^{\rm in}$ and {\em out-rooted} at $S^{\rm out}$.
Given a rooted graph  ${\bf G}=(G,S^{\rm in},S^{\rm out})$, we define 
$\rrev({\bf G})=(G,S^{\rm out},S^{\rm in})$.

Given a rooted graph  $(G,S^{\rm in},S^{\rm out})$, where 
$$S^{\rm in}=\{v^{\rm in}_{1},\ldots,v^{\rm in}_{|S^{\rm in}|}\}\ \ \ \mbox{and}
\ \ \ S^{\rm out}=\{v^{\rm out}_{1},\ldots,v^{\rm out}_{|S^{\rm out}|}\},$$
we define its {\em enhancement} as the graph  $\rep(G,S^{\rm in},S^{\rm out})$ obtained from 
$G$ after adding two 
vertices $u^{\rm in}$
 and  $u^{\rm out}$ and  the edges  in the sets 
$$E^{\rm in}=\{\{v^{\rm in}_{1},u^{\rm in}\},\ldots,\{v^{\rm in}_{|S^{\rm in}|},u^{\rm in}\}\}$$
and  
$$E^{\rm out}=\{\{v^{\rm out}_{1},u^{\rm out}\},\ldots,\{v^{\rm out}_{|S^{\rm out}|},u^{\rm out}\}\}.$$
From now on, we will refer to the vertices $u^{\rm in}, u^{\rm out}$
as the {\em vertex extensions} of $\rep(G,S^{\rm in},S^{\rm out})$ 
and the edge sets $E^{\rm in}$ and $E^{\rm out}$ 
as the {\em edge extensions} of $\rep(G,S^{\rm in},$ $S^{\rm out})$.

\subsection{An extension of the connected search game.}

In the above setting we assumed that searchers cannot make their first move 
in the graph before the fugitive makes his first move. Let $G$ be a graph 
and let $S^{\rm in}, S^{\rm out}\subseteq V(G)$.
A  {\em $(S^{\rm in},S^{\rm out})$-complete strategy for  $G$} is 
a search strategy ${\cal S}$ on $\rep(G,S^{\rm in},S^{\rm out})$ such that
\begin{itemize}
\item[(i)] $E({\cal S},i)=E^{\rm in}$, for some $i$,
\item[(ii)] $E({\cal S},i)\cap E^{\rm out}=\emptyset$, for every  $i$ and  
\item[(iii)] $E({\cal S},i)=E(G)\setminus E^{\rm out}$, for some $i$,
\end{itemize}
where $E^{\rm in}, E^{\rm out}$ are the {\em edge extensions} of $\rep(G,S^{\rm in},$ $S^{\rm out})$.

Based on the above definitions, we define  ${\bf ms}(G,S^{\rm in},S^{\rm out})$ as the minimum
 mixed search number over all  possible $(S^{\rm in},S^{\rm out})$-complete search strategies  for  it. 
 Similarly, we define  $\ms(G,S^{\rm in},S^{\rm out})$ and  $\cms(G,S^{\rm in},S^{\rm out})$ where, 
 in the case of connected searching, we additionally demand that $S^{\rm in}$ induces a connected 
 subgraph of $G$.
 Notice that ${\bf ms}(G)={\bf ms}(G,\emptyset,\emptyset)$ and  that this equality also holds for 
 $\ms$ and $\cms$.

\subsection{Expansions.}

Given a graph $G$ and a set $F\subseteq E(G)$, we define 
$$\partial_{G}(F)=(\bigcup_{e\in F}e)\cap(\bigcup_{e\in E(G)\setminus F}e)$$

Let $G$ be a graph and let $E_1$ and $E_2$ be subsets of $E(G)$.
An {\em $(E_1,E_2)$-expansion} of $G$ is an ordering ${\cal E}=\langle A_{1},\ldots,A_{r}\rangle$
where 
\begin{itemize}
\item[1.] For $i\in\{1,\ldots,r-1\}$, $E_1\subseteq A_{i}\subseteq E(G)\setminus E_2$.
\item[2.] For $i\in\{1,\ldots,r-1\}$, $|A_{i+1}\setminus A_{i}|\leq 1$.
\item[3.] $A_{1}=E_1$,
\item[4.] $A_{r}=E(G)\setminus E_2$.
\end{itemize}

An {\em $(E_1,E_2)$-expansion} of $G$ is {\em connected} if the following  condition holds:
\begin{itemize}
\item[5.] For $i\in\{1,\ldots,r\}$, $G[A_{i}]$ is connected.
\end{itemize}

An {\em $(E_1,E_2)$-expansion} of $G$ is {\em monotone} if the following  condition holds:

\begin{itemize}
\item[6.] $A_{1}\subseteq\cdots \subseteq A_{r}$.
\end{itemize}

Let $i\in\{1,\ldots,r-1\}$.
 The {\em cost} of an expansion ${\cal E}$ at position $i$ 
is defined as ${\bf cost}_{G}({\cal E},i)=|\partial_{G}(A_{i})|+q_i$
where  $q_i$ is equal to one if one of the following holds

\begin{itemize}
\item  $|A_{i}|\geq 2$ and $A_i\setminus A_{i-1}$ contains a pendant edge of $G$
\item  $A_{i}$ consists of only one  edge that  is an isolated edge  of $G$.
\end{itemize}

If none of the above two conditions hold then $q_i$ is equal to $0$.
The {\em cost} of the expansion ${\cal E}$, denoted as ${\bf cost}_{G}({\cal E})$, is the maximum cost 
of  ${\cal E}$ among all positions $i\in\{1,\ldots,r-1\}$.

We define ${\bf p}(G,S^{\rm in},S^{\rm out})$ as the minimum cost  that an 
$(E^{\rm in},E^{\rm out})$-expansion of 
$\rep(G,S^{\rm in},S^{\rm out})$
may have, where $E^{\rm in},E^{\rm out}$ are the edge extensions of $\rep(G,$ $S^{\rm in}$, 
$S^{\rm out})$. We also define 
${\bf mp}(G,S^{\rm in},S^{\rm out})$ 
(if we consider only monotone $(E^{\rm in},E^{\rm out})$-expansions) and 
${\bf cmp}(G,S^{\rm in},S^{\rm out})$  
(if we consider connected monotone 
$(E^{\rm in},E^{\rm out})$-expansions). We finally define $\cmp(G)=\cmp(G,\emptyset,\emptyset)$.

\begin{lemma}
\label{emptymake}
Let $(G,S^{\rm in},S^{\rm out})$ be a rooted graph and let $S_{1}^{\rm in}\subseteq S^{\rm in}$  and 
$S_{1}^{\rm out}\subseteq S^{\rm out}$, {where $G[S^{in}]$ is a connected subgraph of G}. 
Then $\cmp(G,S_1^{\rm in},S_1^{\rm out})\leq \cmp(G,S^{\rm in},S^{\rm out})$.
\end{lemma}

\begin{proof}
Let $E^{\rm in},E^{\rm out}$ and $E^{\rm in}_1,E^{\rm out}_1$ be the edge extensions of 
$\rep(G,S^{\rm in},S^{\rm out})$
 and $\rep(G,S^{\rm in}_1,S^{\rm out}_1)$ respectively. Notice that, as 
 $S_{1}^{\rm in}\subseteq S^{\rm in}$ and 
 $S_{1}^{\rm out}\subseteq S^{\rm out}$, $E_{1}^{\rm in}\subseteq E^{\rm in}$ and 
 $E_{1}^{\rm out}\subseteq E^{\rm out}$.

Let ${\cal E}=\langle A_1, \ldots, A_r\rangle$ be an monotone and connected
$(E^{\rm in},E^{\rm out})$-expansion of 
$\rep(G,S^{\rm in},S^{\rm out})$, with cost at  most $k$. 
\medskip

As $G[S^{\rm in}]$ is a connected subgraph of $G$, for every vertex of 
$S^{\rm in}\setminus S_{1}^{\rm in}$ there exist a path connecting it with a vertex of 
$S_{1}^{\rm in}$ that only uses vertices of $S^{\rm in}$. We define the following edge sets:

\begin{itemize}
\item $E_{1}^{1}$ contains all edges that have a vertex of $S_{1}^{\rm in}$ and a vertex of 
$S^{\rm in}\setminus S_{1}^{\rm in}$ as endpoints. Let 
$V_1=(\bigcup_{e\in E_{1}^{1}}e)\setminus  S_{1}^{\rm in}$, then  $E_{1}^{2}$ contains all 
edges that have both endpoints in $V_1$.
\item $E_{j}^{1}$ contains all edges that have a vertex of $V_{j-1}$ and a vertex of 
$S=S^{\rm in}\setminus (S_{1}^{\rm in}\cup(\bigcup_{l=1\ldots,j-1}V_{l}))$ as endpoints. 
Let $V_j=(\bigcup_{e\in E_{j}^{1}}e)\setminus  S$, then  $E_{j}^{2}$ contains all edges that have 
both endpoints in $V_j$.
\end{itemize}

For each edge set $E_{j}^{i}$, $1\leq j\leq d$ and $i\in\{1,2\}$, where $d$ is the maximum 
distance between a vertex of $S^{\rm in}\setminus S_{1}^{\rm in}$ to some vertex in 
$S_{1}^{\rm in}$, we define arbitrarily an edge ordering $L_{j}^{i}$. We then define an 
ordering ${\cal E}_1$ of edge sets as follows:

\begin{itemize}
\item $A_{1}^{'}=(A_{1}\setminus E^{\rm in})\cup E_{1}^{\rm in}$
\item $A_{1+l}^{'}=A_{1+l-1}^{'}\cup \hat{A}_l$ for $l=1,\ldots, |E_{1}^{1}|$, where 
$\langle\hat{A}_1,\ldots, \hat{A}_{|E_{1}^{1}|}\rangle={\bf prefsec}(L_{1}^{1})$
\item $A_{1+|E_{1}^{1}|+l}^{'}=A_{1+|E_{1}^{1}|+l-1}^{'}\cup \hat{A}_l$ for 
$l=1,\ldots, |E_{1}^{2}|$, where $\langle\hat{A}_1,\ldots, \hat{A}_{|E_{1}^{2}|}\rangle=
{\bf prefsec}(L_{1}^{2})$
\item $A_{1+|E_{1}^{1}|+|E_{1}^{2}|+\cdots +|E_{j-1}^{1}|+ |E_{j-1}^{2}|+l}^{'}=
A_{1+|E_{1}^{1}|+|E_{1}^{2}|+\cdots +|E_{j-1}^{1}|+ |E_{j-1}^{2}|+l-1}^{'}\cup \hat{A}_l$ 
for $l=1,\ldots, |E_{j}^{1}|$, where $\langle\hat{A}_1,\ldots, \hat{A}_{|E_{j}^{1}|}\rangle=
{\bf prefsec}(L_{j}^{1})$
\item $A_{1+|E_{1}^{1}|+|E_{1}^{2}|+\cdots +|E_{j-1}^{1}|+ |E_{j-1}^{2}|+|E_{j}^{1}|+l}^{'}=
A_{1+|E_{1}^{1}|+|E_{1}^{2}|+\cdots +|E_{j-1}^{1}|+ |E_{j-1}^{2}|+|E_{j}^{1}|+l-1}^{'}$ $\cup\ \hat{A}_l$ 
for $l=1,\ldots, |E_{j}^{2}|$, where $\langle\hat{A}_1,\ldots, \hat{A}_{|E_{j}^{2}|}\rangle=
{\bf prefsec}(L_{j}^{2})$
\end{itemize}

Let $s=|E_{1}^{1}|+|E_{1}^{2}|+\cdots +|E_{d}^{1}|+ |E_{d}^{2}|$. Notice that there exist a 
$l_0\in \{2,\ldots,r\}$ such that $A_{1+s}^{'}=(A_{l_{0}}\setminus E^{\rm in})\cup E_{1}^{\rm in}$. 
We define  a second ordering ${\cal E}_2$ of edge sets as follows:  
$A_{1+s+l}^{'}=(A_{l_{0}+l}\setminus E^{\rm in})\cup E_{1}^{\rm in}$ for $l=1,\ldots,r-l_{0}$.\smallskip

Clearly ${\cal E}'={\cal E}_{1}\oplus{\cal E}_{2}$ satisfies conditions 1--4 and therefore is an 
$(E^{\rm in}_1,E^{\rm out}_1)$-expansion of 
$\rep(G,S^{\rm in}_1,S^{\rm out}_1)$. Moreover, the monotonicity and connectivity of $\cal E'$
follows from the monotonicity and connectivity of $\cal E$.\smallskip

Notice that, for every $i\in\{1,\ldots, r\}$, $\partial_{G}(A_i)=\partial_{G}((A_{i}\setminus E^{\rm in})
\cup E_{1}^{\rm in})$
therefore ${\bf cost}_{G}({\cal E}',i)\leq{\bf cost}_{G}({\cal E},i)$. From this we conclude that 
${\cal E}'$ has cost at most $k$.
\end{proof}

Let ${\bf G}_{1},\ldots,{\bf G}_{r}$ be rooted graphs such that 
${\bf G}_i=(G_i,S_{i}^{\rm in},S^{\rm out}_{i})$ where 
$V(G_{i})\cap V(G_{i+1})=S^{\rm out}_{i}=S^{\rm in}_{i+1}$, $i\in\{1,\ldots,r-1\}$.
We define, ${\bf glue}({\bf G}_{1},\ldots,{\bf G}_{r})$ $=(G_1\cup \cdots \cup G_r,S_{1}^{\rm in},S^{\rm out}_{r})$.

\begin{lemma}
\label{glue}
Let ${\bf G}_{1},\ldots,{\bf G}_{r}$ be rooted graphs such that 
${\bf G}_i=(G_i,S_{i}^{\rm in},S^{\rm out}_{i})$ where 
$V(G_{i})\cap V(G_{i+1})=S^{\rm out}_{i}=S^{\rm in}_{i+1}$, $i\in\{1,\ldots,r-1\}$. Then 
$$\cmp({\bf glue}({\bf G}_{1},\ldots,{\bf G}_{r}))\leq \max\{\cmp({\bf G}_{i})\mid i\in\{1,\ldots,r\}\}.$$
\end{lemma}

\begin{proof}
Let $E_i^{\rm in},E_i^{\rm out}$ be the edge extensions and ${\cal E}_i=\langle A_1^i, \ldots, 
A_{l_i}^i\rangle$ 
be an monotone and connected $(E_i^{\rm in},E_i^{\rm out})$-expansion of 
$\rep(G_i,S_i^{\rm in},S_i^{\rm out})$, 
for every $i\in\{1, \dots, r\}$.
Clearly ${\cal E}=\langle A_1^1, \ldots, A_{l_1}^1, A_{l_1}^1\cup A_2^2, \ldots, A_{l_1}^1
\cup A_ {l_2}^2, \ldots, 
(\cup_{1\leq i<r} A_{l_i}^i)\cup A_2^r, \ldots, (\cup_{1\leq i<r} A_{l_i}^i)$ $\cup A_{l_r}^r\rangle$ is an 
$(E_{1}^{\rm in},E^{\rm out}_{r})$-expansion of  ${\bf glue}({\bf G}_{1},\ldots,{\bf G}_{r})$ and, 
as expansions ${\cal E}_i,\ i\in\{1, \dots, r\}$ are monotone and connected, conditions  5 and 6 hold. 

We observe that ${\bf cost}_{G_1\cup \cdots \cup G_r}({\cal E})\leq \max\{{\bf cost}_{G_1}({\cal E}_1),
\ldots, 
{\bf cost}_{G_r}({\cal E}_r)\}$, therefore $\cmp$ $({\bf glue}({\bf G}_{1},\ldots,{\bf G}_{r}))\leq \max
\{\cmp({\bf G}_{i})
\mid i\in\{1,\ldots,r\}\}$.
\end{proof}

\begin{lemma}
\label{equivalence}
For every graph $G$, ${\bf cmms}(G,S^{\rm in},S^{\rm out})={\bf cmp}(G,S^{\rm in},S^{\rm out})$.
\end{lemma}

\begin{proof}
Assume that $G^*=\rep(G,S^{\rm in},S^{\rm out})$ has a complete search strategy ${\cal  S}$ 
satisfying conditions 
(i) -- (iii) with cost at most $k$.
We construct an edge ordering of $E(G)$ as follows.
Observe that, because of the monotonicity of ${\cal S}$,
$E^{(i)}=E({\cal S},i)\setminus E({\cal S},i-1)$. 
For every $i\in\{1,\ldots,|{\cal S}|\}$, we define $L_{i}$
by taking any ordering of the set $E^{(i)}$
and insisting that, if $E^{(i)}$ contains some sliding edge, this edge will be 
the first edge of $L_{i}$. Let ${\cal E}=\langle A_{0},\ldots,A_{r}\rangle$ be  the sequence of 
prefixes of 
$L_{1}\oplus \ldots \oplus L_{|S|}$, including the empty set (that is $A_{0}=\emptyset$). Notice that, 
because of Condition~(i), $A_{s}=E^{\rm in}$
for some $s\in\{1,\ldots,|{\cal S}|\}$,
and,  because of Condition~(iii), $A_{t}=E(G)\setminus E^{\rm out}$, for some 
$t\in\{1,\ldots,|{\cal S}|\}$. We now claim that 
${\cal E}'=\langle A_{s},\ldots,A_{t}\rangle$
is an $(E^{\rm in},E^{\rm out})$-expansion of $G^{*}$.
Indeed, Condition~(1) holds because of Condition (ii)
and Conditions (2) -- (4) hold because of the construction of ${\cal E}'$.
Moreover, the connectivity and the monotonicity of ${\cal E}'$ follow directly from the connectivity 
and the monotonicity of ${\cal S}$.

It remains to prove that the cost of 
${\cal E}'$ is at most $k$. 
For each $j\in\{0,\ldots,|{\cal E}'|\}$
we define $i_{j}$ such that  the unique edge in $A_{j}\setminus A_{j-1}$
is an edge in $E^{(i_j)}$
and we define $h_j$ such that $A_{h_j}\setminus A_{h_j-1}$ contains the fist edge 
of $L_{i_j}$.
Notice now  that the cost of ${\cal E}'$
at positions $h_j$ to $j$ 
is upper bounded by the cost of ${\cal E}'$
at position $h_j$. Therefore, it is enough 
to prove that the cost of ${\cal E}'$
at position $h_j$ is at most $k$.
Recall that this cost is equal to 
$|\partial_{G}(A_{h_j})|+q_{h_j}$.
We distinguish two cases:\\

\noindent{\em Case 1.} If  $q_{h_j}=0$, then 
the cost of ${\cal E}'$
at position $h_j$ is equal to $|\partial_{G}(A_{h_j})|$.
As ${\cal S}$ is monotone, all vertices in $\partial_{G}(A_{h_j})$
should be occupied by searchers after the $i_{j}$-th move of ${\cal S}$
and therefore the cost of ${\cal E}'$
at position $h_j$ is at most $k$.\\

\noindent{\em Case 2.} If $q_{h_j}=1$, then the $i_{j}$-th move of ${\cal S}$
is either the placement of a searcher on a pendant vertex  $x$
or the sliding of a searcher along a pendant edge $\{y,x\}$ towards its pendant vertex $x$.  
In both cases, $x\not\in \partial_{G}(A_{h_j})$
and all vertices in $\partial_{G}(A_{h_j})$ should be occupied by searchers
after the $i_{j}$-th move. In the first case, 
there are in total at least $|\partial_{G}(A_{h_j})|+1$ searchers on the graph
and we are done. In the second case, we observe that, because 
of monotonicity,  
$\partial_{G}(A_{h_j})=\partial_{G}(A_{h_j-1})\setminus \{y\}$.
As after the $(h_j-1)$-th move all vertices of $\partial_{G}(A_{h_j-1})$ were occupied by searchers, 
we obtain that $|\partial_{G}(A_{h_j})|\leq k-1$
and thus the cost of ${\cal E}'$
at position $h_j$ is at most $k$.\\

Now assume that there exist  a monotone and connected  $(E^{\rm in},E^{\rm out})$-expansion of 
$G^*$, say ${\cal E}=\langle A_{1},\ldots,A_{r}\rangle$, with cost at most 
$k$. We can additionally assume that 
${\cal E}$ is properly monotone; this can be done by discarding additional repetitions of a set in 
${\cal E}$.

Moreover, starting from $\cal E$, we can construct a monotone and connected 
$(E^{\rm in},E^{\rm out})$-expansion of $G^*$, with cost at most $k$,
say ${\cal E}'=\langle A'_1,\ldots,A'_r\rangle$, with the following additional property:\\

\noindent{\em Expansion property:} For every $i\in\{1,\ldots,r-1\}$ for which 
$V(A'_i)\subset V(A'_{i+1})$, $A_i'$ contains all edges of $G^*$ with both endpoints in $V(A'_i)$.\\

This can be accomplished by a series of appliances of the following rule:\\

\noindent{\em Rule:} Let $V(A_i)\subset V(A_{i+1})$ for some $i$ and  let 
$L=\langle e_1,\ldots,e_n \rangle$ 
be an ordering of the edges 
$E(G^*)\setminus A_i$ with both endpoints in $V(A_i)$. For every $j\leq i$ define $A'_j=A_j$. 
Then, define
$A'_{i+1}=A_{i}\cup\{e_1\}$, $A'_{i+2}=A_{i}\cup\{e_1,e_2\}$ and so on until 
$A'_{i+n}=A_{i}\cup\{e_1,\ldots, e_n\}$. 
Finally, for every $j\geq i+n$, define $A'_j=A_j\cup\{e_1,\ldots, e_n\}$.\\

One can easily check that, after every application of this Rule, the constructed sequence of edge 
sets is indeed an 
$(E^{\rm in},E^{\rm out})$-expansion of $G^*$ and  furthermore it is monotone and connected. 
Notice that, for $j=1,\ldots, n$, $\partial_{G^*}(A'_{i+j})\subseteq\partial_{G^*}(A_i)$ and  for 
$j\geq i+n$, $|\partial_{G^*}(A'_j)|\leq|\partial_{G^*}(A_j)|$. {Moreover, if $|A_{i}|\geq 2$ and 
$A_i\setminus A_{i-1}$ contains a pendant edge of $G^*$ then  for every $j\in\{1,\ldots, n\}$, 
$|A'_{i+j}|\geq 2$ and $A'_{i+j}\setminus A'_{i+j-1}$ contains the same pendant edge of $G^*$, 
hence the cost of ${\cal E}'$ is at most $k$  (notice that if  $A_{i}$ consists of only one  edge that  
is an isolated edge  of $G^*$
then there does not exist an edge in $E(G^*)\setminus A_i$ with both endpoints in $V(A_i)$, 
therefore we do not need to apply this rule).}

For the rest of the proof, we will consider that the Expansion property holds for the given 
$(E^{\rm in},E^{\rm out})$-expansion 
of $G^*$.

Our target is to define a $(S^{\rm in},S^{\rm out})$-complete monotone search strategy 
$\cal S$ of $G^*$ with cost at most $k$.
 
The first $|S^{\rm in}|$ moves of $\cal S$ will be {\sf p}$(u^{\rm in})$ and the next $|S^{\rm in}|$ will  
be {\sf s}$(u^{\rm in},v_i^{\rm in})$. We denote this sequence of 
moves by ${\cal S}_0$. Notice that $E({\cal S}, 2|S^{\rm in}|)=A_1$.

For every vertex $u$ in the set $V^*= V(G^*)\setminus S^{\rm in}\setminus \{u^{\rm out}\}$, we define 
$l_u$ to be the first integer in $\{1,\ldots,r\}$ such that $u\in V(A_{l_u})$.

 Let $L=\langle u_1,\ldots,u_{|V^*|} \rangle$ be an ordering of $V^*$ such that $i\leq j$ when 
 $l_{u_i}\leq l_{u_j}$. 
 Notice that, for each $i\in\{1,\ldots,|V^*|\}$,  the vertex $u_{i}$ is an endpoint of the unique 
 edge $e_{i}$
in $A_{l_{u_{i}}-1}\setminus A_{l_{u_{i}}}$ and let $v_{i}$ be the other endpoint of $e_{i}$.
Notice that, because of the connectivity and the monotonicity of ${\cal E}$,  
$v_{i}\in  \partial_{G^{*}}(A_{l_{u_{i}}-1})$.
We also observe that $u_{i}$ is pendant iff $u_{i}\not\in\partial_{G^{*}}(A_{l_{u_{i}}})$.
We define $E'=\{e_{1},\ldots,e_{|V^{*}|}\}$ and  we call a set $A_j$, $j\in\{1,\ldots, r\}$, crucial iff 
$|A_{j-1}\cap E'|<|A_j\cap E'|$.

For each  $i\in\{1,\ldots,|V^*|\}$, we define a sequence 
${\cal S}_{i}$ of moves as follows:  If $v_{i}\in  \partial_{G^{*}}(A_{l_{u_{i}}})$
then the first move of ${\cal S}_{i}$ is ${\sf p}(u_i)$, otherwise it is ${\sf s}(v_{i},u_{i})$. The 
rest of the moves in ${\cal S}_{i}$ are the removals, one by one,  of the
searchers in $\partial_{G^{*}}(A_{l_{u_{i}}-1})\setminus \partial_{G^{*}}(A_{l_{u_{i}}})$. 
Then we define ${\cal S}={\cal S}_{0}\oplus {\cal S}_{1}\oplus\cdots \oplus {\cal S}_{|V^{*}|}$.

Notice that, according to the Expansion property,  all edges of the sets $A_j$, for $j=1,\ldots,l_{u_1}$ 
have both endpoints 
in $S^{\rm in}$. Moreover, for every $i\in\{1,\ldots,|V^*|-1\}$ all edges of the sets $A_j$, for $j=l_{u_i},
\ldots,l_{u_i+1}-1$, 
have both endpoints in $V(A_{l_{u_i}})$ and  all edges of the sets $A_j$, for $j=l_{u_{|V^*|}},\ldots,r$, 
have both endpoints in $V(A_{l_{u_{|V^*|}}})$. 

First we show that the following claim is true:\\

\noindent{\em Claim 1.} For  every $A_{j}$, $j\in\{1,\ldots,r\}$, the vertices of $\partial_{G^{*}} (A_{j})$
are exactly the vertices occupied by searchers after the last move of ${\cal S}_{m_j}$,
 where $m_j$ is the index of the edge in 
$(A\cap E')\setminus(A_{j-1}\cap E')$, where $A$ is the first crucial set of $\cal E$ such that 
$A_j\subseteq A$. \\
   
Clearly, this is true for $A_1=E^{\rm in}$. Assume that it holds for $A_{j'}$.

We will show that the vertices in $\partial_{G^*}(A_{j'+1})$ 
are exactly the vertices occupied by searchers after the last move of ${\cal S}_{m_{j'+1}}$.
  
If $A_{j'+1}$ is not crucial then $\partial_{G^*}(A_{j'+1})\subseteq\partial_{G^*}(A_{j'})$ and  
$m_{j'+1}=m_{j'}$, 
therefore Claim 1 holds.

Now, if $A_{j'+1}$ is crucial and  $\{e_{m_{j'+1}}\}=(A_{j'+1}\cap E')\setminus(A_{j'}\cap E')$, then 
$v_{m_{j'+1}}\in \partial_{G^*}(A_{j'})$ and therefore must be occupied by a searcher. 
We distinguish three cases:\\
   
\noindent{\em Case 1.} If $v_{m_{j'+1}}\in \partial_{G^*}(A_{j'+1})$ and $u_{m_{j'+1}}\in 
\partial_{G^*}(A_{j'+1})$, 
then $ \partial_{G^*}(A_{j'+1})= \partial_{G^*}(A_{j'})\cup\{u_{m_{j'+1}}\}$ and  the first move in 
${\cal S}_{m_{j'+1}}$ 
will be  ${\sf p}(u_{m_{j'+1}})$. \\

\noindent{\em Case 2.} If $v_{m_{j'+1}}\in \partial_{G^*}(A_{j'+1})$ and $u_{m_{j'+1}}\not\in 
\partial_{G^*}(A_{j'+1})$ 
then $\partial_{G^*}(A_{j'+1})=\partial_{G^*}(A_{j'})$.\\

\noindent{\em Case 3.} If $v_{m_{j'+1}}\not\in \partial_{G^*}(A_{j'+1})$, 
then $ \partial_{G^*}(A_{j'+1})= (\partial_{G^*}(A_{j'})\setminus\{v_{m_{j'+1}}\})\cup\{u_{m_{j'+1}}\}$,
and the first move in ${\cal S}_{m_{j'+1}}$ will be ${\sf s}(v_{m_{j'+1}},u_{m_{j'+1}})$.\\

\noindent Observe that in all three cases the Claim 1 holds.\\

Let $V_{\cal S}(i)$ be the set of vertices already visited by searchers after the $i$-th move of $\cal S$, 
and let $V_{\cal S}=\langle V_{\cal S}(1),\ldots,V_{\cal S}(r)\rangle$. 
Notice that this sequence is monotone and  that if the $i$-th move belong to the subsequence 
${\cal S}_j$, 
then $V_{\cal S}(i)=V(A_{l_{u_j}})$. We must next prove the following claim:\\

\noindent{\em Claim 2:} For every $i\in\{1,\ldots,|{\cal S}|\}$, all edges of $G^*[V_{\cal S}(i)]$ 
are clean.\\

Clearly, the claim is true for $i\in\{1,\ldots, 2\cdot |S^{\rm in}|\}$. Assume that it holds for some 
$i\in{2\cdot |S^{\rm in}|+1,\ldots,r}$, we will show that all edges of $G^*[V_{\cal S}(i+1)]$ are clean. 
We must distinguish three cases about the $(i+1)$-th move:\\

\noindent{\em Case 1.} It is a removal, say {\sf r}$(u)$. Notice that 
$G^*[V_{\cal S}(i+1)]=G^*[V_{\cal S}(i)]$, 
therefore the Claim will not be true if {\sf r}$(u)$ is a recontamination move. In this case,
there exist an edge connecting $u$ with a vertex not in $V_{\cal S}(i)$, say $v$. As 
$u\in \partial_{G^{*}}(A_{l_{u_{j}}-1})\setminus \partial_{G^{*}}(A_{l_{u_{j}}})$, 
for some $j\in\{1,\ldots,|V^*|\}$, 
all edges with $u$ as 
endpoint must belong to $A_{l_{u_{j}}}$, therefore $\{u,v\}\in A_{l_{u_{j}}}$. 
But $V_{\cal S}(i)=V(A_{l_{u_{j}}})$, a contradiction. \\

\noindent{\em Case 2.} It is a placement of searcher say {\sf p}$(u)$. By the definition of $\cal S$, 
there exist an 
edge $\{u,v\}$, where $v$ is a vertex in $V_{\cal S}(i)$. Notice that, according to our search game, 
all such 
edges are clean after  {\sf p}$(u)$, thus all edges of $G^*[V_{\cal S}(i+1)]$ are clean.\\

\noindent{\em Case 3.} It is a slide, say {\sf s}$(v_j,u_j)$, for some $j\in\{1,\dots,|V^*|\}$. 
As in the previous case,
$G^*[V_{\cal S}(i+1)]$ contains all edges of 
$G^*[V_{\cal S}(i)]$ and additional all edges with $u_j$ as the first endpoint and  a vertex 
$v\in V_{\cal S}(i)$ as the other. 
According to our search game, after the $i$-th move there must be searcher in $v_j$, 
therefore due to Claim 1, 
$v_j\in\partial_{G^*}(A_{l_{u_{j}}})$. Notice that, the Claim will not be true if {\sf s}$(v_j,u_j)$ is a 
recontamination move,
i.e., there exist an edge connecting $v_j$ with a vertex, say $u$, not in $V_{\cal S}(i)=V(A_{l_{u_j}})$.  
As 
$v_j\not\in \partial_{G^{*}}(A_{l_{u_{j}}})$, 
all edges with $u$ as 
endpoint must belong to $A_{l_{u_{j}}}$, therefore $\{v_j,u\}\in A_{l_{u_{j}}}$, a contradiction. \\

In all three cases we show that after the $(i+1)$-th move of $S$ all edges of $G^*[V_{\cal S}(i+1)]$
are clean, 
therefore Claim 2 is true.\\

We will now prove that $\cal S$ is a $(S_{1},S_{2})$-complete strategy for $G^*$. 
Clearly, Condition~(i) holds for 
every strategy starting with ${\cal S}_{0}$. Moreover, Condition~(ii) holds as 
$v^{\rm out}$ is not a vertex of $V^{*}$  and 
therefore, no placement on $u^{\rm out}$ or sliding towards $u^{\rm out}$ appears in ${\cal S}$. 
Notice that, according to Claim 2, for every $i\in\{1,\ldots,|V^*|-1\}$, $E({\cal S},|{\cal S}_0\oplus
\cdots\oplus{\cal S}_{i-1}|+1)=
\cdots=E({\cal S},|{\cal S}_0\oplus\cdots\oplus{\cal S}_{i-1}|+|{\cal S}_i|)=A_{l_{u_{i+1}}-1},$ 
and that for $i=|V^*|$, $E({\cal S},|{\cal S}_0\oplus\cdots\oplus{\cal S}_{|V^*|}|)=A_r$, therefore
Condition~(iii) holds. 

By the definition of $\cal S$, it is clear that $\cal S$ is a connected search strategy, moreover, 
according to Claim 2, $\cal S$ is monotone. 
It remains to prove that $\cal S$ has cost at most $k$. For the first $2|S^{\rm in}|$ moves, we use 
$|S^{\rm in}|={\bf cost}_{G^*}({\cal E},1)\leq k$ searchers.  Assume that after $j$ moves exactly $k$
searchers are 
occupying vertices of $G^*$ and  that the $(j+1)$-th move is ${\sf p}(u_i)$, for some 
$i\in\{i,\ldots,|V^*|\}$. Then the 
vertices in $\partial_{G^{*}}(A_{l_{u_{i}}-1})$ are exactly the vertices occupied by the $k$ searchers, 
therefore 
$|\partial_{G^{*}}(A_{l_{u_{i}}-1})|=k$. Observe that, if $u_i$ is not pendant, then
$\partial_{G^{*}}(A_{l_{u_{i}}})=\partial_{G^{*}}(A_{l_{u_{i}}-1})\cup\{u_i\}$, therefore 
$|\partial_{G^{*}}(A_{l_{u_{i}}})|=k+1$, 
a contradiction and  if $u_i$ is pendant then  $\partial_{G^{*}}(A_{l_{u_{i}}})=\partial_{G^{*}}
(A_{l_{u_{i}}-1})$ and  
the cost of $\cal E$ at position $l_{u_{i}}$ is $|\partial_{G^{*}}(A_{l_{u_{i}}})|+1=k+1$, 
again a contradiction. 
Thus, for every move of $\cal S$ at most $k$ searchers are occupying vertices of $G^*$.
\end{proof}

\subsection{Contractions.}
\label{cont8e}

Let $(G_{1},S^{\rm in}_{1},S^{\rm out}_{1})$ and $(G_{2},S^{\rm in}_{2},S^{\rm out}_{2})$ 
be  rooted graphs. We say that $(G_{1},S^{\rm in}_{1},$ $S^{\rm out}_{1})$ is 
{\em a contraction} of  $(G_{2},S^{\rm in}_{2},S^{\rm out}_{2})$ and  we denote this 
fact 
by $(G_{1},S^{\rm in}_{1},S^{\rm out}_{1})$ $\preceq (G_{2},S^{\rm in}_{2},S^{\rm out}_{2})$ 
if there exist a surjection $\phi: V(G_{2})\rightarrow V(G_{1})$ such that:

\begin{itemize}
\item[${\bf 1.}$] for every vertex $ v\in V(G_{1})$,  $G_{2}[\phi^{-1}(v)]$ is connected
\item[${\bf 2.}$] for every two distinct vertices $u,v\in V(G_{1})$, it holds that 
$\{v,u\}\in E(G_{1})$ if and only if the graph $G_{2}[\phi^{-1}(v)\cup 
\phi^{-1}(u)]$ is connected
\item[${\bf 3.}$] $\phi(S_{2}^{\rm in})= S_{1}^{\rm in}$
\item[${\bf 4.}$] $\phi(S_{2}^{\rm out})= S_{1}^{\rm out}$
\end{itemize}
We also write $(G_{1},S^{\rm in}_{1},S^{\rm out}_{1})\preceq_{\phi} (G_{2},S^{\rm in}_{2},
S^{\rm out}_{2})$ 
to make clear the function that certifies the contraction relation.
We say that $G_{1}$ is a {\em contraction}
of $G_{2}$ if $(G_{1},\emptyset,\emptyset)\preceq (G_{2},\emptyset,\emptyset)$
and we denote this fact by $G_{1}\preceq G_{2}$. If furthermore $G_1$ is not isomorphic to 
$G_2$ we say that 
$G_1$ is a {\em proper contraction} of $G_2$.

We define the {\em minor} relation for the two rooted graph by removing  in the second property 
the demand that if $\{u,v\}\notin E(G_1)$ then $G_{2}[\phi^{-1}(v)\cup 
\phi^{-1}(u)]$ is not connected.  We denote the minor relation by $(G_{1},S^{\rm in}_{1},
S^{\rm out}_{1})$ $\leq (G_{2},S^{\rm in}_{2},S^{\rm out}_{2})$.  Again, we say that $G_{1}$ is a 
{\em minor}
of $G_{2}$ if $(G_{1},\emptyset,\emptyset)\leq (G_{2},\emptyset,\emptyset)$
and we denote this fact by $G_{1}\leq G_{2}$.

\begin{lemma}
\label{lem:clos}
If $(G_{1},S^{\rm in}_{1},S^{\rm out}_{1})$ and $(G_{2},S^{\rm in}_{2},S^{\rm out}_{2})$ 
are  rooted graphs and  $(G_{1},S^{\rm in}_{1},$ 
$S^{\rm out}_{1})\preceq (G_{2},S^{\rm in}_{2},S^{\rm out}_{2})$,
then 
${\bf cmp}(G_{1},S^{\rm in}_{1},S^{\rm out}_{1}))\leq {\bf cmp}(G_{2},S^{\rm in}_{2},S^{\rm out}_{2}))$.
\end{lemma}

\begin{proof}
Suppose that ${\cal E}=\langle A_{1},\ldots,A_{r}\rangle$ is a monotone 
$(E_{2}^{\rm in },E_{2}^{\rm out})$-expansion 
of $G_{2}^{*}=\rep(G_2,S^{\rm in}_2,S^{\rm out}_2)$ with cost at most $k$.
Our target is to construct a monotone $(E_{1}^{\rm in },E_{1}^{\rm out})$-expansion 
of $G_{1}^{*}=\rep(G_1,S^{\rm in}_1,S^{\rm out}_1)$ with cost at most $k$.

Let  $\phi$ be a function where
$(G_{1},S^{\rm in}_{1},S^{\rm out}_{1})\preceq_{\phi} (G_{2},S^{\rm in}_{2},S^{\rm out}_{2})$. 
We  consider an extension $\psi$ of $\phi$ 
that additionally maps $u_{2}^{\rm in}$ to $u_{1}^{\rm in}$
and $u_{2}^{\rm out}$ to $u_{1}^{\rm out}$. Notice that the construction of $\psi$ yields the following:
$$(G_{1}^{*},S^{\rm in}_{1}\cup \{u_{1}^{\rm in}\},S^{\rm out}_{1}\cup \{u_{1}^{\rm out}\})
\preceq_{\phi} (G_{2}^{*},S^{\rm in}_{2}\cup 
\{u_{2}^{\rm in}\},S^{\rm out}_{2}\cup \{u_{2}^{\rm out}\})$$

Given an edge $f=\{x,y\}\in E(G_{1})$ we consider the set $E_f$
containing all edges of $G_{2}$ with one endpoint in $\psi^{-1}(x)$ and one 
endpoint in $\psi^{-1}(y)$. We now pick, arbitrarily,
an edge in $E_{f}$ and we denote it by $e_{f}$. We also set $E'=\{e_f\mid f\in E(G_{1})\}$. 
Then it is easy to observe that 
${\cal E}'=\langle A_{1}\cap E',\ldots,A_{r}\cap E'\rangle$ is a connected expansion of $G_{1}^*$ 
and that 
the cost of ${\cal E}'$ at step $i$ is no bigger than the cost of ${\cal E}$
at the same step, where $i\in\{1,\ldots,r-1\}$.
\end{proof}

\begin{lemma}
\label{lem:cs-clos}
If $G_{1}$ and $G_{2}$ 
are  two graphs and  $G_{1}\preceq G_{2}$,
then 
${\bf cs}(G_{1})\leq {\bf cs}(G_{2})$.
\end{lemma}

\begin{proof}
First observe that if this is the case any contraction of $G_2$ can be derived by applying a finite
number of edge-contractions of some edges in $E(G_2)$. 

It suffices to prove that the Lemma hold if $G_1$ is obtained by the contraction of edge 
$e=\{u,v\}\in E(G_1)$ to vertex $x_{uv}$. Let $\cal S$ be a connected search strategy for 
$G_2$ that in any step uses at most $k$ searchers. Based on $\cal S$ we will construct a 
search strategy ${\cal S}'$ for $G_1$.  Let $i$ be an integer in $\{1,\ldots,|{\cal S}|\}$. 
We distinguish eight cases:\\

\noindent {\em Case 1}: If the $i$-th move of $\cal S$ is {\sf p}$(x)$ for some vertex 
$x\notin\{u,v\}$ then the next move of ${\cal S}'$ will be {\sf p}$(x)$.\\

\noindent {\em Case 2}: If the $i$-th move of $\cal S$ is {\sf r}$(x)$ for some vertex 
$x\notin\{u,v\}$ then the next move of ${\cal S}'$ will be {\sf r}$(x)$.\\

\noindent {\em Case 3}: If the $i$-th move of $\cal S$ is {\sf s}$(x,y)$ for some vertices 
$x,y\notin\{u,v\}$ then the next move of ${\cal S}'$ will be {\sf s}$(x,y)$.\\

\noindent {\em Case 4}: If the $i$-th move of $\cal S$ is {\sf p}$(u)$ or  {\sf p}$(v)$ then 
the next move of ${\cal S}'$ will be {\sf p}$(x_{uv})$.\\

\noindent {\em Case 5}: If the $i$-th move of $\cal S$ is {\sf r}$(u)$ or  {\sf r}$(v)$ then 
the next move of ${\cal S}'$ will be {\sf r}$(x_{uv})$.\\

\noindent {\em Case 6}: If the $i$-th move of $\cal S$ is {\sf s}$(z,u)$ or {\sf p}$(z,v)$ for 
some vertex $z$ then the next move of ${\cal S}'$ will be {\sf s}$(z,x_{uv})$.\\

\noindent {\em Case 7}: If the $i$-th move of $\cal S$ is {\sf s}$(u,z)$ or {\sf p}$(v,z)$ 
for some vertex $z$ then the next move of ${\cal S}'$ will be {\sf s}$(x_{uv},z)$.\\

\noindent {\em Case 8}: If the $i$-th move of $\cal S$ is {\sf s}$(u,v)$ or {\sf p}$(v,u)$ then 
the next move of ${\cal S}'$ will be defined according the lateral cases from the 
$(i+1)$-th move of $\cal S$.\\

Observe that ${\cal S}'$ is a complete search strategy for $G_1$. Furthermore, as 
$\cal S$ is connected, ${\cal S}'$ must also be connected. Finally, it is clear that 
${\cal S}'$ at any step uses at most $k$ searchers, thus $\cs(G_1)\leq k$.
\end{proof}

\subsection{Parameters and obstructions.} 

We denote by ${\cal G}$ the class of all 
graphs. A \textit{graph parameter} is a function $f: {\cal G}\rightarrow \mathbb{N}$. 
Given a graph parameter $f$ and an integer $k\in\mathbb{N}$ we define the \textit{graph class}
$\mathcal{G}[f,k]$, containing all the graphs $G\in {\cal G}$ where $f(G)\leq k$.

Let ${\cal H}$ be a graph class. We denote by $\obs({\cal H})$ the set of all graphs 
in ${\cal G}\setminus {\cal H}$ that are minimal with respect to the relation $\preceq$.

\subsection{Cut-vertices and blocks.}

We call the 2-connected 
components of a graph $G$ {\em blocks}. If the removal of an edge 
in a graph  increases the number of its connected components then it is called {\em bridge}.
We consider the subgraph of $G$ induced by the endpoints of a bridge 
of $G$  as one of its blocks and we call it {\em trivial block} 
of $G$. 

A {\em cut-vertex} of a graph $G$ is a vertex such that $G\setminus x$ has more connected 
components than $G$.
Given a graph $G$ a {\em cut-vertex of a block $B$} of $G$ is a cut-vertex of $G$ 
that belongs in $V(B)$.

Let $G$ bet a graph and let $x\in V(G)$. We define $${\cal C}_{G}(x)=\{(x,G[V(C)\cup
\{x\}])\mid \mbox{$C$ is a connected component of $G\setminus x$}\}.$$ 

 Let $B$ be a
block of $G$ and let $x$ be a cut-vertex of $B$. We denote by $C_{G}(x,B)$ the (unique) graph in 
${\cal C}_{G}(x)$  that contains $B$ as a subgraph and by $\overline{\cal C}_{G}(x,B)$ the graphs in 
${\cal C}_{G}(x)$  that do not contain $B$.

\subsection{Outerplananr graphs.}

We call a graph $G$ \textit{outerplanar} if it can be embedded in the plane such that all
its vertices are incident to its infinite face (also called {\em outer face}). This  embedding, 
when exists, is unique
up to homeomorphism and, from now on, each outerplanar graph is accompanied with  such an 
embedding.  
An edge $e\in E(G)$ is called \textit{outer edge of $G$},
if it is incident to the outer face of $G$, otherwise is called a \textit{chord of $G$}.

A face $F$ of an outer planar graph that is different than the outer face, is called \textit{haploid} 
if and only if at most  one edge
incident to $F$ is a chord, otherwise $F$ is a \textit{inner} face. A vertex $u\in V(G)$
is \textit{haploid} if  it is incident to an haploid face and  \textit{inner} if it is incident
to an inner face (notice that some vertices can be both inner and haploid). 
A vertex of $G$ that is not inner or haploid is called {\em outer}.
We call a chord {\em haploid} if it is incident to an haploid face. Non-haploid 
chords are called {\em internal chords}. 

\begin{figure}
\begin{center}
 \scalebox{0.8}{\includegraphics{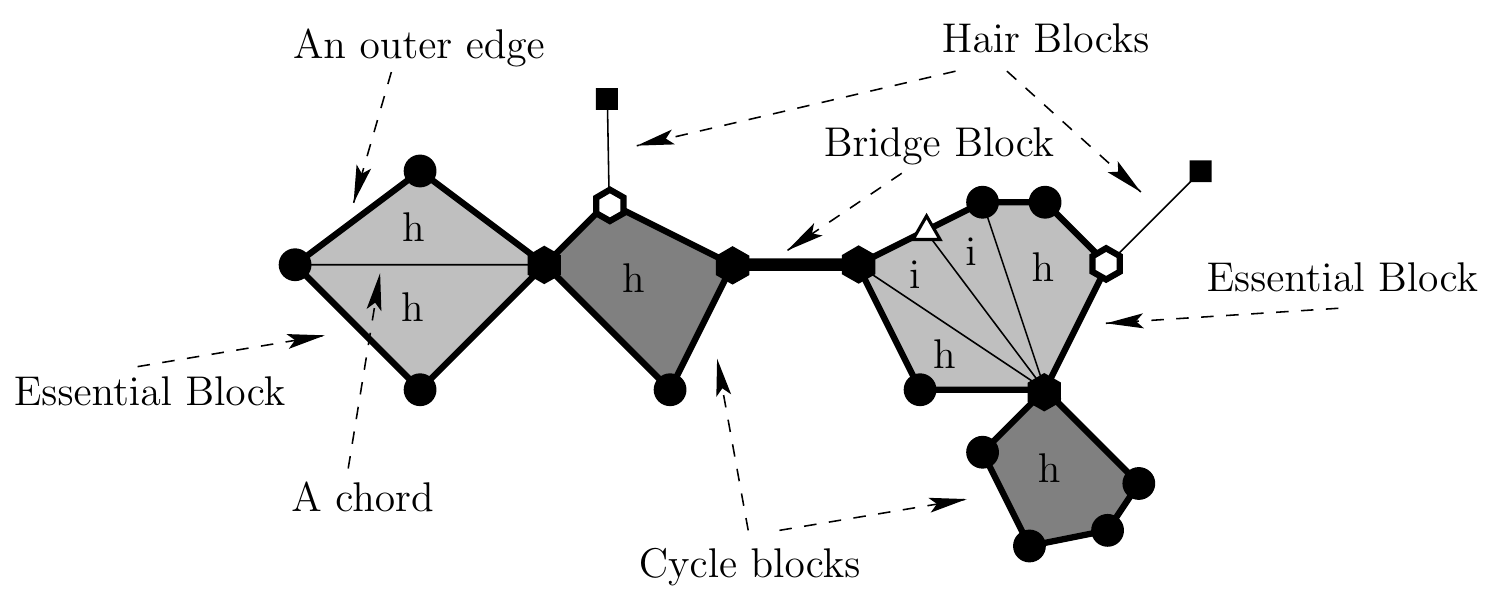}}
 \end{center}
\caption{A outerplanar graph and its blocks. The cut-vertices are hexagonal and the outer vertices 
are squares. 
Inner and haploid faces are denoted by ``i" and ``h" respectively.
There are, in total, four inner vertices (all belonging to the  essential block on the right) and, 
among them only the
triangular one is not  an haploid vertex. The white hexagonal vertices are the light cut-vertices 
while the rest of the 
hexagonal vertices are the heavy ones.}
 \label{FanRootObs}
\end{figure}

\begin{observation}
A block of a connected outerplanar graph with more than one edge can be one of the following.
\begin{itemize}
\item a {\em hair block}: it is a trivial block containing exactly one vertex of degree 1 in $G$.
\item a {\em bridge block}: it is a trivial block that is not a hair-block.
\item a {\em cycle block}: if it is a chordless non-trivial block, or
\item an {\em essential block}: if it is a non-trivial block with at least one chord.
\end{itemize}
\end{observation}

Let $G$ be a connected outerplanar graph with more than one edges.
Given a  cut-vertex $c$ of $G$, 
we say that $c$ is {\em light} if it is the (unique) cut-vertex of exactly one hair block.
If a cut-vertex of $G$ is not light then it is {\em heavy}. 

It is known that the
\textit{class of outerplanar graphs} is closed under the relations $\leq,\preceq$ and
that a graph is outerplanar if and only if $K_4\nleq G$ and $K_{2,3}\nleq G$.

\begin{figure}
\begin{center}
 \scalebox{0.8}{\includegraphics{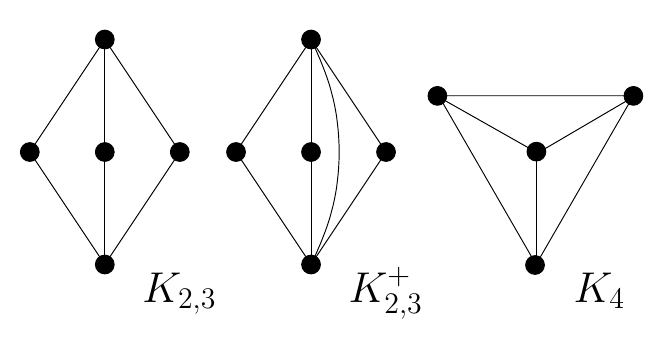}}
 \end{center}
\caption{The set ${\cal O}_{1}$.}
\label{FanRsootObs}
\end{figure}

Let $K_{2,3}^{+}$ be the graph obtained by $K_{2,3}$ after connecting the two vertices of
degree 3 (Figure~\ref{FanRsootObs}).

\begin{lemma}
\label{lema:outerplanar} 
If ${\cal H}$ is the class of all outerplanar graphs, then $\obs({\cal 
H})={\cal O}_{1}$.
\end{lemma}

\begin{proof}
Observe that the graphs in ${\cal O}_{1}$ 
cannot be embedded in the plane in such  a way that all of its vertices are incident to a single 
face and therefore neither the graphs in ${\cal O}_{1}$, neither the graphs that contain as a 
contraction 
a graph in ${\cal O}_{1}$, can be outerplannar. 

To complete the proof, one must show that every non-outerplannar graph can be contracted to 
a graph in 
${\cal O}_{1}$. Let $G$ be non-outerplannar, then $K_4\leq G$ or $K_{2,3}\leq G$. Clearly, 
as $K_{4}$ is a clique, 
$K_{4}\leq G$ implies that $K_{4}\preceq G$. 
Suppose now that $K_{2,3}\leq G$. Let $V_{x},V_{y}, V_{1}, V_{2}, V_{3}$ be the 
vertex sets of the connected subgraphs of $G$ that are contracted towards creating the 
vertices of $K_{2,3}$ ($V_{x}$ and $V_{y}$ are contracted to vertices of degree 3). If there is no 
edge in 
$G$ between two vertices in $V_{a}$ and $V_{b}$ for 
some $(a,b)\in \{(x,y),(1,2),(2,3),(1,3)\}$ then $K_{2,3}\preceq G$. If the only such edge 
is between $V_{x}$ and $V_{y}$ then $K_{2,3}^{+}\preceq G$ and in any other case, 
$K_{4}\preceq G$. 
\end{proof}

\section{Obstructions for Graphs with $\cms$ at most 2} 
\label{sec:2}

In this section we give the obstruction set for graphs with connected monotone mixed search 
number at most 2 and we prove its correctness.

\subsection{The obstruction set for $k=2$}

\hskip0.6cm Let ${\cal D}^1={\cal O}_{1}\cup \cdots \cup {\cal O}_{12}^{}$ where ${\cal O}_{1}$
is depicted in Figure~\ref{FanRsootObs}, ${\cal O}_{2},\ldots,{\cal O}_{9}$ are depicted in 
Figure~\ref{FanRossotsObs}
and ${\cal O}_{10}$ and ${\cal O}_{11}$ and ${\cal O}_{12}$ are constructed as follows.

\begin{description}
\item[${\cal O}_{10}:$]  contains every graph that can be constructed by taking three 
disjoint copies of some graphs in Figure~\ref{FanRootObs}
and then identify the vertices denoted by $v$ in each of them to a single vertex. 
There are, in total, 35 graphs generated in this way.
\item[${\cal O}_{11}:$]  contains every graph that can be constructed by taking two 
disjoint copies of some graphs in Figure~\ref{direction}
and then identify the vertices denoted by $v$ in each of them to a single vertex.   
There are, in total, 78 graphs generated in this way.
\item[${\cal O}_{12}:$]   contains every graph that can be constructed by taking 
two disjoint copies of some graphs in Figure~\ref{if9ton}
and then identify the vertices denoted by $v$ in each of them to a single vertex.    
There are, in total, 21 graphs generated in this way.
\end{description}

Observe that, ${\cal D}^{1}$ contains 177 graphs.

\subsection{Proof strategy}

\begin{lemma}
\label{sub_dir}
${\cal D}^{1}\subseteq \obs(\mathcal{G}[\cmp,2])$.
\end{lemma}

\begin{proof}
From Lemma~\ref{lem:clos},
it is enough to check that for every $G\in{\cal D}^1$, the following two conditions 
are satisfied ({\bf i}) $\cmp({G})\geq 3$ and ({\bf ii}) for every edge $e$ of $G$ it holds that
$\cmp(G\slash e)\leq 2$. One can verify that this is correct by inspection,
as 
this  concerns only a finite amount of graphs and, for each of them, there exists 
a finite number of edges to contract.
\end{proof}

\begin{lemma}
\label{sup_dir}
${\cal D}^{1}\supseteq \obs(\mathcal{G}[\cmp,2])$. 
\end{lemma}

The rest of this section is devoted to the proof of  Lemma~\ref{sup_dir}.
For this, our strategy is to consider the set 
$${\cal Q}=\obs(\mathcal{G}[\cmp,2])\setminus {\cal D}^{1}$$
and prove that  ${\cal Q}=\emptyset$  (Lemma~\ref{tll0o}). For this, we need a series of 
structural results whose
proofs use the following three fundamental properties of the set ${\cal Q}$.

\begin{lemma}
\label{goinaldirs}
Let $G\in {\cal Q}$. Then the following hold.
\begin{itemize}
\item[i.]  $\cmp(G)\geq  3$.
\item[ii.] If $H$ is a  proper contraction of $G$, then $\cmp(G)\leq 2$.
\item[iii.] $G$ does not contain any of the graphs in    ${\cal D}^{1}$ as a contraction.
\end{itemize}
\end{lemma}

\begin{proof}
Properties i. and ii. hold because $G\in \obs(\mathcal{G}[\cmp,2])$.
For property iii. suppose, to the contrary, that $G$ contains some graph in $H\in {\cal D}^{1}$ 
as a contraction.
From Lemma~\ref{sub_dir}, $H\in \obs(\mathcal{G}[\cmp,2])$. Clearly, $H$ is different than $G$
as ${\cal Q}$ does not contain members of ${\cal D}^{1}$. Therefore, $H$ is a proper contraction 
of $G$ and, from property ii., 
$\cmp(H)\leq 2$. This contradicts to the fact that $H\in \obs(\mathcal{G}[\cmp,2])$ and thus 
$\cmp(H)\geq 3$.
\end{proof}

\begin{figure}[h]
\begin{center}
 \scalebox{.6}{\includegraphics{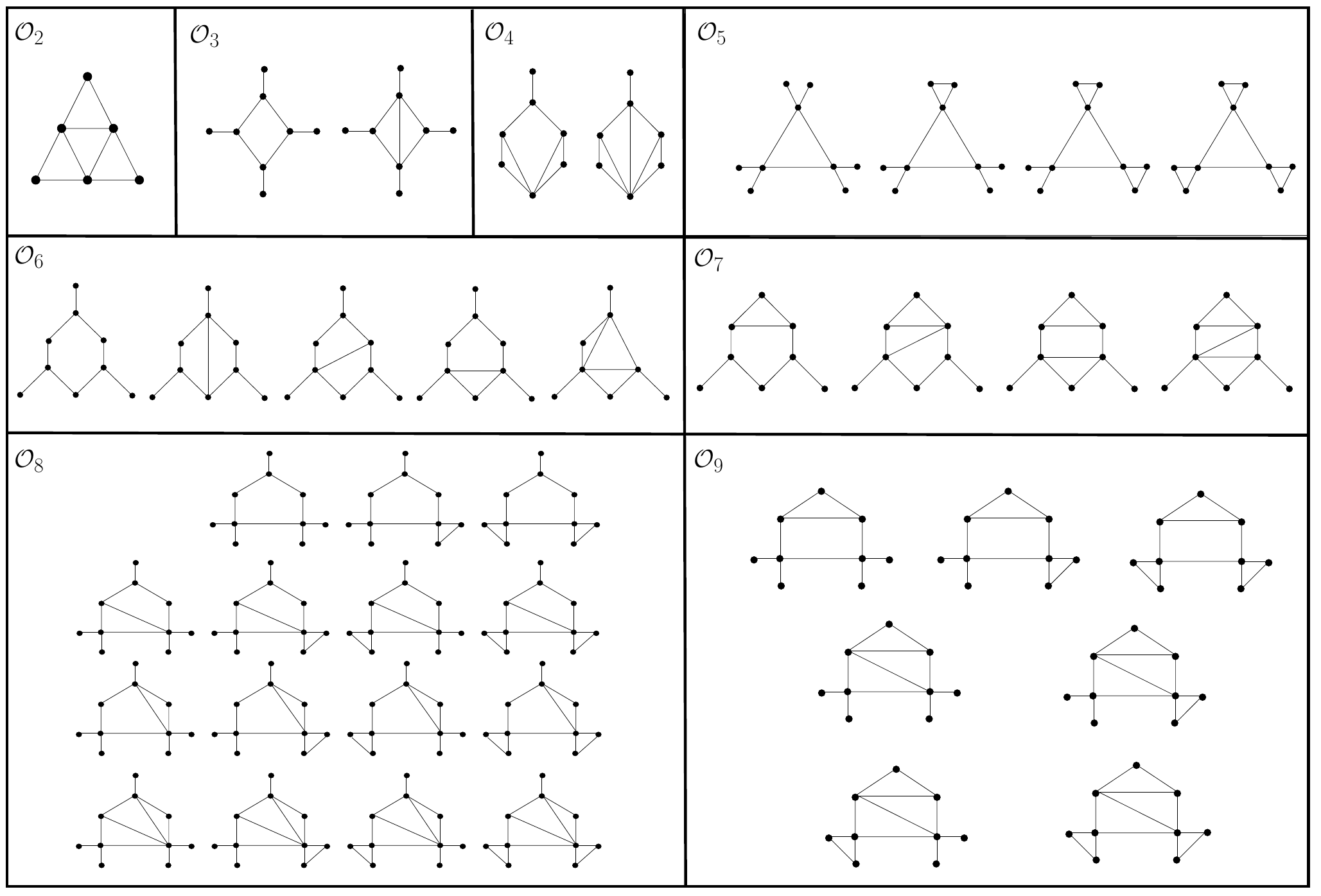}}
 \end{center}
\caption{The sets of graphs in ${\cal D}_1$.}
\label{FanRossotsObs}
\end{figure}

\subsection{Basic structural properties}

\begin{lemma}
\label{basic_structure} 
Let $G\in {\cal Q}$. The following hold:
\begin{itemize}
\item[1.] $G$ is outerplanar. 
\item[2.] Every light cut-vertex of $G$ has degree at least 3. 
\item[3.] Every essential block $B$ of $G$, has exactly two haploid faces. 
\item[4.] Every block of $G$, has at most 3 cut-vertices
\item[5.] Every  cut-vertex of a non-trivial block of $G$ is an haploid vertex. 
\item[6.] Every block of $G$ contains at most 2 heavy cut-vertices. 
\item[7.] If a block of $G$ has 3 cut-vertices, then there are two, say $x$ and $y$,  of these vertices
that are not both heavy and  are connected by an haploid edge. 
\item[8.] If an essential block of $G$ with haploid faces $F_{1}$ and $F_{2}$ has two heavy cut-vertices, then 
one can choose one, say $c_{1}$, of these two heavy cut-vertices so that it is incident to $F_{1}$ and 
one say $c_{2}$ that is incident to $F_{2}$. 
Moreover, this assignment can be done in such a way that if there is a third light cut-vertex $c_{3}$, adjacent to one, say $c_{1}$, of $c_{1},c_{2}$, then
$c_{3}$ is incident to $F_{1}$ as well.
\end{itemize}
\end{lemma}

\begin{proof} 1. By the third property of Lemma~\ref{goinaldirs}, $G$ cannot be contracted 
to a graph in ${\cal O}_1$ 
and  therefore, from Lemma~\ref{lema:outerplanar}, $G$ must be outerplannar.

2. Let $c$ be a light cut-vertex of a block $B$ in $G$, with degree 2 (notice that, as $c$ is a 
cut-vertex, $c$ cannot 
have degree 1 or 0). That means that $c$ belongs to a path with at least two edges, the hair 
block $B$ and an edge 
say $e$. Observe that $\cmp(G/B)=\cmp(G)$, contradicting to the second property of 
Lemma~\ref{goinaldirs}.

3.  Let $B$ be an essential block of $G$. As it is essential, it has at least one chord, therefore
it has at least 2 haploid faces. Assume, that $B$ has at least 3 haploid faces. Choose 3 of them, 
say $F_1, F_2$ and $F_3$ (see Figure~\ref{helpinproof1}). 
Let $S\subseteq E(B)$ be the set of all chords incident to $B$.  Contract in $G$ 
all edges in $E(G)\setminus S$ not belonging to those faces. Then, for each of the three faces, 
contract all but 
two edges not in $S$ that are incident to  $F_1, F_2$ and $F_3$  and  notice that the obtained 
graph is the graph in 
${\cal O}_2$, a contradiction to the third property of Lemma~\ref{goinaldirs}. 

\begin{figure}[h]
\begin{center}
\includegraphics{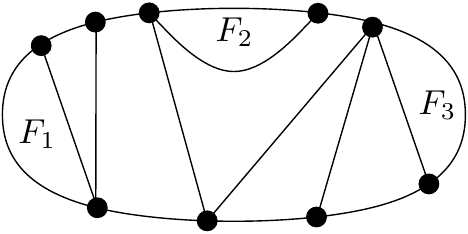}\end{center}
\caption{An example for the proof of Lemma~\ref{basic_structure}.3.}
\label{helpinproof1}
\end{figure}

4. Let $B$ be a block of $G$ containing more than 3 cut-vertices. Chose four of them, 
say $c_1,c_2,c_3$ and $c_4$. 
Let $S\subseteq E(B)$ be the set of all chords incident to $B$ (see Figure~\ref{helpinproof2}). 
Contract all edges in $E(G)\setminus S$  not having an 
endpoint in $\{c_1,c_2,c_3,c_4\}$. Then, contract all edges $e\in E(B)\setminus S$ such that 
$e\nsubseteq\{c_1,c_2,c_3,c_4\}$ 
and  all edges not in $E(B)$, except from one for each of the cut-vertices. Notice that the obtained 
graph belongs to ${\cal O}_3$, 
a contradiction to the third property of Lemma~\ref{goinaldirs}.

\begin{figure}[h]
\begin{center}
\includegraphics{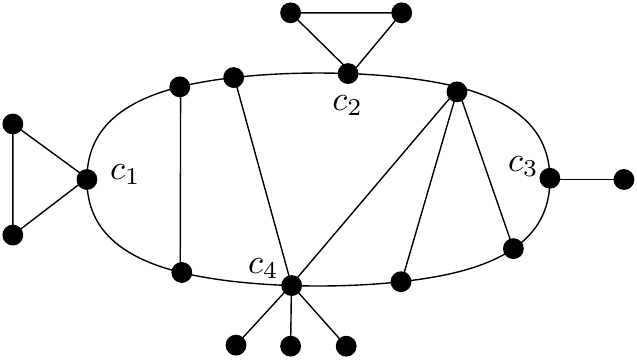}\end{center}
\caption{An example for the proof of Lemma~\ref{basic_structure}.4.}
\label{helpinproof2}
\end{figure}

5. Let $B$ be a block of $G$ containing a cut-vertex $c$ that is not haploid and  let 
$S\subseteq E(B)$ be the set of 
all chords incident to $B$ (see Figure~\ref{helpinproof3}). Contract all edges in 
$E(G)\setminus E(B)$ not having $c$ as endpoint and  all edges in 
$E(B)\setminus S$ not having $c$ as endpoint, except from two edges for each of the haploid faces. 
Then contract 
all edges not in $E(B)$ with $c$ as endpoint, except for one. Notice that the obtained graph 
belongs to ${\cal O}_4$, 
a contradiction to the third property of Lemma~\ref{goinaldirs}.

\begin{figure}[h]
\begin{center}
\includegraphics{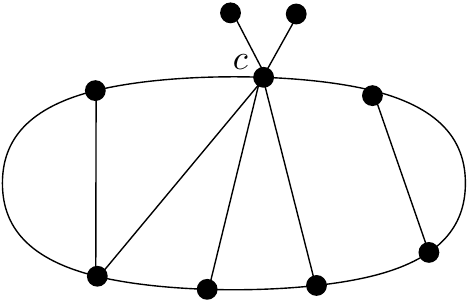}\end{center}
\caption{An example for the proof of Lemma~\ref{basic_structure}.5.}
\label{helpinproof3}
\end{figure}

6. Let $B$ be a block of $G$ containing three heavy cut-vertices, say $c_1,c_2$ and $c_3$ 
(see Figure~\ref{helpinproof4}). 
We contract all edges in $B$ except from 3 so that $B$ is reduced to a triangle $T$ with 
vertices $c_1,c_2$ and $c_3$. 
Then, in the resulting graph $H$, for each $c_{i}, i\in\{1,2,3\}$, 
in ${\cal C}_{H}(c_{i})\setminus \{T\}$ contains either a non trivial block or at least two hair blocks.
In any case, $H$ can be further contracted to one of the graphs in ${\cal O}_{5}$ a 
contradiction to the third property of Lemma~\ref{goinaldirs}.

\begin{figure}[h]
\begin{center}
\includegraphics{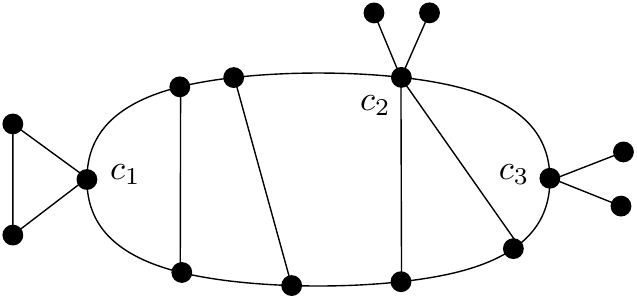}\end{center}
\caption{An example for the proof of Lemma~\ref{basic_structure}.6.}
\label{helpinproof4}
\end{figure}

7. Let $\{x,y,z\}$ be three cut-vertices of a (not-trivial) block $B$. If no two of them are 
connected by an outer edge, then contract 
all blocks of $G$, except $B$, to single edges, then contract all outer edges of $B$ that do 
not have an endpoint in 
$\{x,y,z\}$ and continue contracting hair blocks with a vertex of degree $\geq 4$, as long as this 
is  possible. This 
creates either a graph in ${\cal O}_{6}$ or a graph that after the contraction of a hair block makes 
a graph in 
${\cal O}_{7}$ or a graph that after the contraction of two hair blocks is a graph makes a graph in  
${\cal O}_{4}$ and, 
in any case, we have a  contradiction to the third property of Lemma~\ref{goinaldirs}.
We contract $G$ to a  graph $H$ as follows: 
\begin{itemize}
\item if for some $w\in\{x,y,z\}$ the set $\overline{\cal C}_{G}(w,B)$ contains at least 
two elements, then
contract  the two of  them to a pendant edge (that will have $w$ as an endpoint)
and the rest of them to $w$.
\item if for some $w\in\{x,y,z\}$  the set $\overline{\cal C}_{G}(w,B)$ contains 
only one element that is not a hair, then contact it to a triangle (notice that this is always 
possible because of (2)).
\end{itemize}
\noindent{\em Case 1}. $|V(B)|\in\{3,4\}$. Then because of (6), one, say $x$ of $\{x,y,z\}$ is non-heavy 
and there  is an outer edge connecting $x$ with one, say $y$, vertex in $\{x,y,z\}$. Then $x,y$ 
is the required pair of vertices.

\noindent{\em Case 2}.   $|V(B)|>4$ and there is at most one outer edge $e$ with endpoints from 
$\{x,y,z\}$ in $H$.
W.l.o.g. we assume that $e=\{x,y\}$. Notice  $e$ is a haploid edge, otherwise $H$ can be 
contracted to 
the 5th graph in ${\cal O}_{6}$. Moreover at least one of  $x$, $y$ is non-heavy, otherwise 
$H$ can be contracted 
to one of the graphs in ${\cal O}_{8}\cup {\cal O}_{9}$ (recall that $B$ may have one or two 
haploid faces).

\noindent{\em Case 3}.  There are two outer edges with endpoints from $\{x,y,z\}$.
W.l.o.g. we assume that these edges are $\{x,y\}$ and $\{y,z\}$. 
One, say $\{x,y\}$, of  $\{x,y\}$, $\{y,z\}$ is haploid, otherwise
 $H$ can be contracted to some graph in ${\cal O}_{4}$. 
If $\{x,y\}$ has a light endpoint, then we are done, otherwise, from (6),
$z$ is light. In this remaining case,  if  if $\{z,y\}$ is haploid, then it is also the required edge, otherwise 
$H$ can be contracted to a graph in ${\cal O}_{9}$.

8. Let $x$ and $y$ be two heavy cut-vertices vertices of $B$. From (5)  $x,y$ are among the 
vertices that 
are incident to the faces $F_{1}$ and $F_{2}$. Suppose, in contrary, that for some face, say 
$F\in\{F_1,F_2\}$, there is no cut vertex in $\{x,y\}$ 
that is incident to $F$. Then $G$ can be contracted to one of the graphs in ${\cal O}_{9}$.
This is enough to prove the first statement except from the case
where $x$ and $y$ are  both lying in both haploid faces and there is a third light cut-vertex $z$
incident to some, say $x$, 
of $x$, $y$.  In this case, $x$ is assigned the face where $z$ belongs and $y$ is assigned to 
the other.
\end{proof}

Let $G\in {\cal Q}$ and let $B$ be a block of $G$. Let also $S$ be the 
set of cut vertices of $G$ that belong to $B$. According to Lemma~\ref{basic_structure},
we can define a rooted graph ${\bf G}_{B}=(B,X,Y)$ such that 
\begin{itemize}
\item $\{X,Y\}$ is a partition of $S$ where $X$ and $Y$ are possibly empty.
\item if $B$ has a chord, then all vertices in $X$ and $Y$ are haploid. 
\item $|X|\leq 1$ and $|Y|\leq 2$.
\item If $|Y|=2$, then its vertices are connected with an edge $e$ and one of them is light and, 
moreover, in the case where $B$ has a chord then $e$ is haploid.
\item If $B$ has a chord, we name the haploid faces of $B$ by $F_{1}$ and $F_{2}$ such that 
all vertices in $X$ are incident to $F_{1}$ and all vertices od $Y$ are incident to $F_{2}$.
\end{itemize}

\begin{lemma}
\label{bolek}
Let $G\in {\cal Q}$ and let $B$ be a block of $G$. Then $\cmp({\bf G}_{B})\leq 2$.
\end{lemma}

\begin{proof}
We examine the non-trivial case where $B$ is a non-trivial block and contains two haploid faces 
$F_{1}$ and $F_{2}$.
As $B$ is 2-connected and outer-planar, all vertices of $V(B)$ belong to the unique hamiltonian 
cycle of $B$, say $C$. 
Our proof is based on the fact that there are exactly two haploid faces and this gives a sense 
of direction on how the search should be performed.
To make this formal, we create an ordering ${\cal A}$ of the edges of $E(B)$ using the f
ollowing procedure.

\begin{tabbing}
1. {\bf if} $X\neq \emptyset$, \= {\bf then} \\
2. \> $Q\leftarrow X$, \\
3. \> {\bf else} \\
4. \>   $Q\leftarrow\{x\}$ \=
where   $x$ is an arbitrarily chosen vertex  \\
 \> \> in the boundary of $F_{1}$.\\
5. $R\leftarrow Q$\\
6. $i\leftarrow 1$\\
7. {\bf while} \= there is a vertex $v$ in $V(B)\setminus R$ that is connected with some, say $u$,   \\
 \> vertex in $Q\setminus Y$ whose unique neighbor in  $V(B)\setminus R$ is $v$, \\
8. \>  $R\leftarrow R\cup\{v\}$\\
9. \>  $Q\leftarrow (Q\setminus \{u\})\cup\{v\}$\\
10. \>  $e_{i}=\{u,v\}$ \\
11. \>  $i\leftarrow i+1$ \\
12. \>  {\bf if} $Q\in E(B)$, \= {\bf then} \\
13. \> \> $e_{i} \leftarrow Q$, \\
14. \> \> $i\leftarrow i+1$\\
15. {\bf if} $Y\in E(B)$, \= {\bf then}\\
16. \> $e_{i} \leftarrow Y$
\end{tabbing}

Let $E^{\rm in},E^{\rm out}$ be the edge extensions of $\rep({\bf G}_B)$, and let  
${\bf prefsec}({\cal A})=\langle A_{0}$, 
$\ldots,A_{r}\rangle$.  It is easy to verify  that 
${\cal E}=\langle E^{\rm in}, A_{0}\cup E^{\rm in} ,\ldots,A_{r}\cup E^{\rm in}\rangle$
is a monotone and connected $(E^{\rm in},E^{\rm out})$-expansion of $\rep({\bf G}_{B})$, 
with cost at most 2.
\end{proof}

\begin{figure}[h]
\begin{center}
 \scalebox{0.6}{\includegraphics{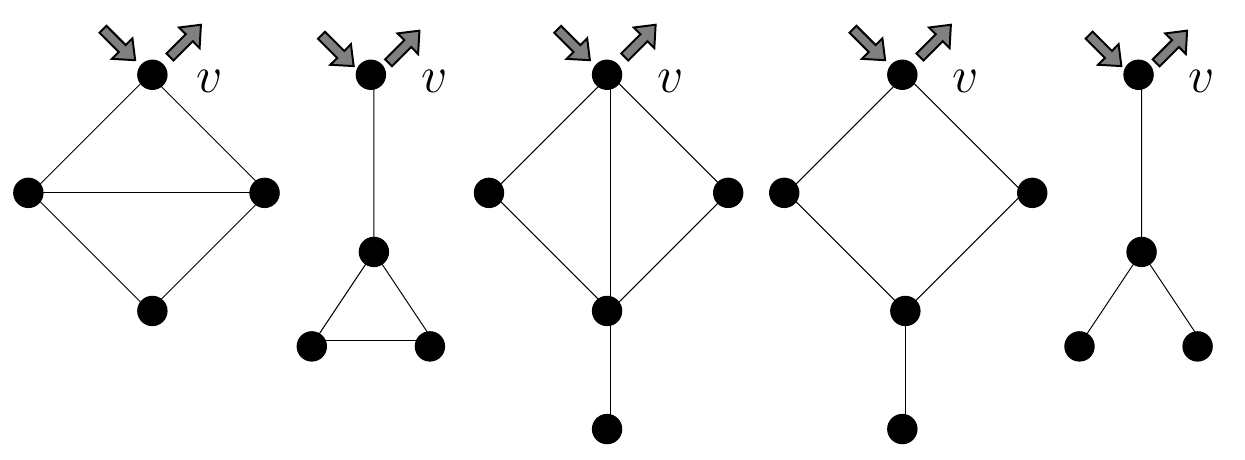}}
 \end{center}
\caption{The set ${\cal A}$ contains 5 $r$-graphs, each
 of the form $(G,\{v\},\{v\})$.}
\label{FanRootObs}
\end{figure}

\subsection{Fans}

Let $G$ be a graph and $v$ be a vertex in $V(G)$. We denote by ${\bf G}^{(v)}$ the rooted graph 
$(G,\{v\},\{v\})$ and we refer to it as the {\em graph $G$ doubly rooted on $v$}.

We say that a graph $G$, doubly rooted on some vertex $v$ is a {\em fan}
if none of the graphs in the set ${\cal A}$ depicted in Figure~\ref{FanRootObs} is a contraction 
of the rooted graph 
${\bf G}^{(v)}$ and $G$ is outerplanar.

\begin{lemma}
\label{seasrchfun}
Let ${\bf G}^{(v)}=(G,\{v\},\{v\})$ be a graph doubly rooted at some vertex $v$. If 
${\bf G}^{(v)}$ is a fan, then 
$\cmp({\bf G}^{(v)})\leq 2$.
\end{lemma}

\begin{proof} 
We claim first that if ${\bf G}^{(v)}$ is a fan, then the graph $G\setminus v$ is a collection 
of paths where each of 
them has at least one endpoint 
that is a neighbour of $v$. Indeed, if this is not correct, then some of the connected components of 
$G\setminus v$ would be contractible to either a $K_{3}$ or a $K_{1,3}$.
In the first case, $G$ is either non-outperlanar or ${\bf G}^{(v)}$ can be contracted to to the 
first two rooted graphs of 
Figure~\ref{FanRootObs}. In the second case $G$ is either non-outeplanar or 
${\bf G}^{(v)}$ the last three graphs of Figure~\ref{FanRootObs}.
Moreover, if both endpoints of a path in the set of connected components $G\setminus v$ 
are non adjacent to $v$ in $G$, then ${\bf G}^{(v)}$ can be contracted to the last rooted 
graph in Figure~\ref{FanRootObs}.

Let now $P_{1},\ldots, P_{r}$ be the connected components of $G\setminus v$
and, for each $i\in\{1,\ldots,r\}$, let $\{v_{1}^{i},\ldots,v_{j_i}^{i}\}$ be the vertices of $P_{i}$,
ordered as in $P_{i}$, such that $v_{1}^{i}$ is adjacent to $v$ in $G$.
Let  $u^{\rm in}$ and $u^{\rm out}$ be the two vertices added in ${\bf enh}({\bf G}^{(v)})$.  
If $e^{\rm in}=\{v,u^{\rm in}\}$ and $e^{\rm out}=\{v,u^{\rm out}\}$ then the edge expansions of 
${\bf enh}({\bf G}^{(v)})$ is $E^{\rm in}=\{e^{\rm in}\}$ and $E^{\rm out}=\{e^{\rm out}\}$.

For each $i\in\{1,\ldots,r\}$ we define the edge ordering 
$${\cal A}_i=\langle \{v,v_1^i\},\{v_1^i, v_2^i\},\{v,v_2^i\},\{v_2^i, v_3^i\}\ldots,\{v_{j_i}^i,v\}\rangle,$$
then we delete from ${\cal A}_i$ the edges not in $E(G)$. Let ${\cal A}_i'$ be the orderings 
obtained after the edge deletions. 
We define ${\cal A}=\langle e^{\rm in} \rangle \oplus {\cal A}_1\oplus\cdots\oplus{\cal A}_r$. 
Notice that ${\bf prefsec}({\cal A})$ is a monotone and connected 
$(E^{\rm in},E^{\rm out})$-expansion of ${\bf enh}({\bf G}^{(v)})$ with cost at most 2. 
Therefore, $\cmp({\bf G}^{(v)})\leq 2$ as required.\end{proof}

\subsection{Spine-degree and central blocks}

Given a graph $G$ and a vertex $v$ we denote by ${\cal C}^{(v)}_{G}$ the set of all 
graphs in  ${\cal C}_{G}(v)$, each doubly rooted on $v$. The {\em spine-degree} of $v$ 
is the number of doubly rooted
graphs in   ${\cal C}^{(v)}_{G}
$ that are not fans.

A cut-vertex of a graph $G$ is called {\em central cut-vertex}, if it has spine-degree greater than 1 
and  a block 
of $G$ is called {\em central block} if it contains at least 2 central cut-vertices.

\begin{lemma}
\label{lem:atmost2nonfuns} 
Let $G\in {\cal Q}$. The following hold:
\begin{itemize}
\item[1.] All vertices of $G$ have spine-degree at most 2. 
\item[2.] None of the blocks of $G$ contains more than 2 central cut-vertices.
\item[3.]  $G$ contains at least one central cut-vertex
\item[4.] There is a total ordering $B_{1},B_{2},\ldots,B_{r}$  ($r\geq 0$)
of the central blocks of $G$ and a total ordering $c_1,\ldots,c_{r+1}$ of the central cut-vertices of 
$G$ such that, for $i\in\{1,\ldots,r\}$, the cut-vertices of $B_{i}$ are $c_{i}$ and $c_{i+1}$.
\end{itemize} 
\end{lemma}

\begin{proof} 1. Let $v$ be a vertex of $G$ with spine-degree at least 3. That means that there exist 
at least 
three subgraphs of $G$, doubly rooted on $v$, that can be contracted to some graph in 
$\cal A$, therefore 
$G$ can be contracted to a graph in ${\cal O}_{10}$, a contradiction.

2. Suppose that $B$ is a block of $G$ containing 3 (or more) central cut-vertices, say
$c_1,c_2$ and $c_3$. 
Construct the graph $H$ by contracting all edges of $B$ to a triangle $T$ with 
$\{c_{1},c_{2},c_{3}\}$ as vertex set.
As $c_{i}$ is a central vertex, there is a rooted graph $R_{i}$ in ${\cal C}_{G}(c_i)$ that contains 
some of the 
graphs in ${\cal A}$ as a rooted contraction. Next we apply the same contractions to $H$ f
or every $c_{i}, i\in\{1,2,3\}$ 
and then contract to vertices all blocks of $H$ different than $T$ and not contained in some $R_{i}$. 
It is easy to 
see that the resulting graph is a graph in ${\cal O}_5$, a contradiction.

3. Assume that $G$ has no central cut-vertices.  We distinguish two cases.

\noindent{\em Case 1}. There is a cut-vertex 
of $G$, say $v$, such that all rooted graphs in ${\cal C}_{G}^{(v)}$ are fans and let 
${\bf G}_{1},\ldots {\bf G}_{r}$
 be these rooted graphs. From  Lemmata~\ref{glue} and~\ref{seasrchfun}, we conclude
 that $\cmp(G,\{v\},\{v\})=\cmp({\bf glue}({\bf G}_{1},\ldots {\bf G}_{r}))\leq 2$
 and from Lemma~\ref{emptymake},  $\cmp(G)\leq 2$ contradicting to the first condition of 
 Lemma~\ref{goinaldirs}. 
 
\noindent{\em Case 2}. For every cut-vertex $v$ of $G$, at least one of the rooted graphs in 
${\cal C}_{G}^{(v)}$ is not a fan. 
We denote by $H_v$ the corresponding non-fan 
rooted graph in  ${\cal C}_{G}^{(v)}$  (this is unique due to the fact that $v$ is non-central). 
Among all cut vertices, 
let $x$ be one for which 
the set $V(G)\setminus V(H_{x})$ is maximal.
Let $B$ be the block of $H_{x}$ that contains $x$ and let $S$ be the set of the cut-vertices 
of $G$ that belong to $B$ (including $x$).

For every $y\in S$ we denote by
${\cal W}_{y}=\{{\bf W}_{y}^{1},\ldots,{\bf W}_{y}^{r_{y}}\}$ the set 
of all rooted graphs in ${\cal C}_{G}^{(y)}$, 
except from the one, call it ${\bf R}_{y}$,  that contains $B$. 
We also define 
${\bf W}_{y}={\bf glue}({\bf W}_{y}^{1},\ldots,{\bf W}_{y}^{r_{y}})$.
We claim that all ${\bf W}_{y}, y\in S$ are fans. Indeed, if for some $y$, ${\bf W}_{y}$
is a non-fan, because $y$ is not central, ${\bf R}_{y}$ should be 
a fan, contradicting the choice of $x$.

According to the above, the edges of $G$ can be partitioned to those of 
the rooted graph ${\bf G}_{B}$ and the edges in the rooted graphs ${\bf W}_{y},\ y\in S$.
Let also ${\bf  G}_{B}=(B,X,Y)$ and we assume that, if $Y=\{y_{1},y_{2}\}$, then $y_{1}$ is light.

 Notice that, according to Lemma~\ref{seasrchfun} $\cmp({\bf W}_{y})\leq2,\ y\in S$ and 
 according to Lemma~\ref{bolek}
 $\cmp({\bf G}_B)\leq 2$. 
 We distinguish two cases.
\smallskip

 \noindent{\sl Case 2.1.} $Y=\{y_{1},y_{2}\}$ where $y_{2}$ is light. Then let  
 ${\bf G}_{1}=(G[\{y_1,y_2\}], \{y_1,y_2\},$ $ \{y_2\})$ and 
 ${\bf G}_{2}=(G[\{y_1,y_2\}], \{y_2\}, \{y_1\})$. Clearly  
 $\cmp({\bf G}_1)=2$ and  $\cmp({\bf G}_2)=1$. Therefore, if 
 ${\bf G}'= {\bf glue}({\bf G}_B, {\bf G}_1, {\bf W}_{y_2}, {\bf G}_2, 
 {\bf W}_{y_1})$, then, from Lemma~\ref{glue},
  $\cmp({\bf G}')\leq 2$.
 
 \noindent{\sl Case 2.2.} $Y=\{y_{1}\}$. Let ${\bf G}'= {\bf glue}({\bf G}_B, {\bf W}_{y_1})$, 
 then, from Lemma~\ref{glue},
  $\cmp({\bf G}')$ $\leq 2$.
\smallskip

In both cases, if $X=\{x\}$, then we set ${\bf G}={\bf glue}({\bf W}_{x},{\bf G}')$ while if 
$X=\emptyset$ we set ${\bf G}={\bf G}'$. In any case, we observe that, from Lemma~\ref{glue},  
$\cmp({\bf G})\leq 2$. 
From Lemma~\ref{emptymake},  $\cmp(G)\leq 2$ contradicting to the first condition of 
Lemma~\ref{goinaldirs}. 

As in both cases we reach a contradiction  $G$ must contain at least one central cut-vertex. 

4.  Let $C$ be the set of all central cut-vertices of $G$.
For each $c\in C$, let ${\cal X}_c$ be the subset of ${\cal C}^{(v)}_{G}$
that contains all its members that are not fans. Clearly, ${\cal X}_c$
contains exactly two elements. Notice that none of the vertices in $C\setminus \{c\}$ 
belongs in the double rooted graphs in ${\cal C}^{(v)}_{G}\setminus {\cal X}_{c}$. Indeed, 
if this is the case for some vertex $y\in C\setminus \{c\}$, then the member of ${\cal C}^{(y)}_{G}$ 
that avoids $c$ would be a subgraph of some member of 
${\cal C}^{(v)}_{G}\setminus {\cal X}_{c}$ and this would imply that some fan would contain as 
a contraction some double rooted graph that is not a fan. We conclude that for each $c\in C$
there is a partition $p(c)=(A_{c},B_{c})$ of  $C\setminus \{c\}$ 
such that that all members of ${A}_{c}$ are vertices of one of the 
members of ${\cal X}_c$ and all members of $B_c$ are vertices of the other.

We say that a vertex $c\in C$ is {\em extremal} 
if $p(c)=\{\emptyset,C\setminus \{c\}\}$

We claim that for any three vertices $\{x,y,z\}$ of $C$, there is one, say $y$ of them such that  
$x$ and $z$ belong 
in different sets of $p(y)$. Indeed, if this is not the case, 
then one of the following would happen: either there is a vertex $w\in C$ such that $x,y,z$ 
belong to different elements of ${\cal C}_{G}^{(w)}$, a contradiction
to the first statement of this lemma or $x,y,$ and $z$ belong to the same block of $G$, a 
contradiction to the second statement of this lemma. 

By the above claim, there is a path $P$ 
containing all central cut-vertices in $C$ and we assume that this path is of minimum length 
which permits us to assume that its endpoints are extremal vertices of $C$. Moreover, 
heavy cut-vertices 
in $V(P)$ are members of $C$. Let $c_{1},\ldots,c_{r+1}$ be the 
central cut-vertices ordered as they appear in $P$.
As, for every $i\in\{1,\ldots,r\}$ there is a block $B_{i}$ 
containing the  central cut-vertices $c_{i}$ and $c_{i+1}$ we end up with the two orderings 
required in the forth statement of the lemma.
\end{proof}

Let $G$ be a graph in   ${\cal Q}$.
Suppose also that $c_1,\ldots,c_{r+1}$ and $B_{1},B_{2},\ldots,B_{r}$ are as in 
Lemma~\ref{lem:atmost2nonfuns}.4.
We define the {\em extremal blocks} of $G$ as follows:
\begin{itemize}
\item If $r>0$, then among all blocks that contain $c_{1}$ as a cut-vertex let $B_{0}$
be the one such that $C_{G}(c_{1},B_{0})$, doubly rooted at $c_{1}$, is not a fan, 
does not contain any edge of the central blocks of $G$ and does not contain $c_{r+1}$.
Symmetrically, among all blocks that contain $c_{r+1}$ as a cut-vertex let $B_{r+1}$
be the one such that $C_{G}(c_{r+1},B_{r+1})$ doubly rooted at $c_{r+1}$ is not a fan, 
does not contain any edge of the central blocks of $G$ and does not contain $c_{1}$. 
\item If $r=0$, then let $B_{0}$ and $B_{1}$ be the two blocks with the property that for $i\in\{0,1\}$, 
 $C_{G}(c_{1},B_{i})$, doubly rooted at $c_{1}$,  is not a fan.
\end{itemize}

 We call $B_{0}$ and $B_{r+1}$ {\em left} and {\em right} extremal block of $G$ respectively.
 We also call the blocks of $G$ that are either central or extremal  {\em spine} blocks of $G$.
Let $A(G)$ be the set of cut-vertices of the graphs  $B_{0},B_{1},B_{2},\ldots,B_{r},B_{r+1}$.
We partition $A(G)$ into three sets $A_{1}$, $A_{2}$ and  $A_{3}$ as follows:

\begin{itemize}
\item $A_{1}=\{c_{1},\ldots,c_{r+1}\}$ (i.e. all central vertices).
\item $A_{2}$ contains all vertices of $A(G)$ that belong to central blocks and are not central vertices.
\item $A_{3}$ contains all vertices of $A(G)$ that belong to extremal blocks and are not central 
cut-vertices.
\end{itemize}
Moreover, we further partition $A_{3}$ to two sets $A_{3}^{(0)}=A_{3}\cap V(B_{0})$ and 
$A_{3}^{(r+1)}=A_{3}\cap V(B_{r+1})$.

\begin{figure}[h]
\begin{center}
 \scalebox{0.85}{\includegraphics{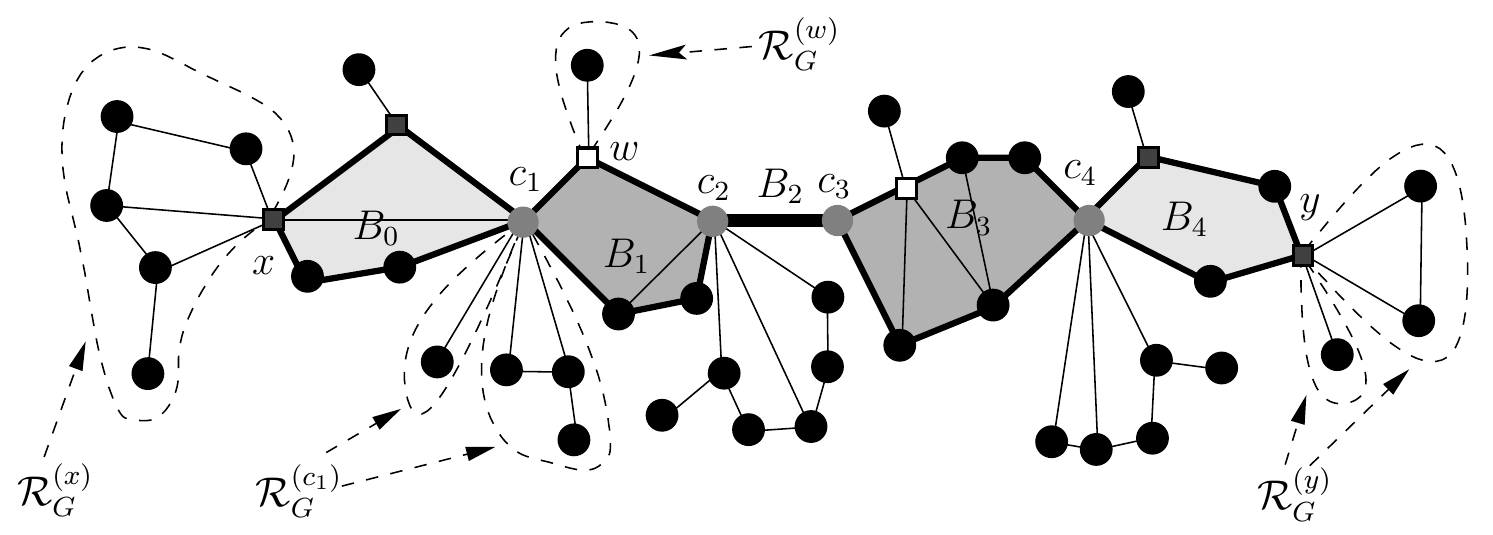}}
 \end{center}
\caption{A graph $G$ and the blocks $B_{0},B_{1},\ldots,B_{3},$ and $B_{4}$.
The cut-vertices in $A_{1}=\{c_{1},\ldots,c_{4}\}$ are the grey circular vertices, the vertices in 
$A_{2}$ are the white square vertices and  the vertices in $A_{3}$
are the dark square vertices.}
\label{FanRoossstObs}
\end{figure}

Let $G\in {\cal Q}$ and $v\in A(G)$. We denote by ${\cal R}_{G}^{(v)}$ the set of  the doubly 
rooted graphs in
${\cal C}_{G}^{(v)}$ that do not contain any of the  edges of the spine blocks of $G$.

\begin{lemma}
\label{lem:interfdans} 
Let $G\in {\cal Q}$ and $v\in A(G)$. 
The following hold:
\begin{itemize}
\item[a)]  Each  doubly rooted graph in ${\cal R}_{G}^{(v)}$ is a fan.
\item[b)] All vertices in $v\in A_{2}$ are light, i.e.,   for each $v\in A_{2}$ 
  ${\cal R}_{G}^{(v)}$ contains 
exactly  one graph that is a hair block of $G$.  
\end{itemize}
\end{lemma}
\begin{proof}
a) Let $v\in A(G)$. We distinguish two cases:\\

\noindent{\em Case 1:} $v\in A_1$. If there exist a double rooted graph in 
${\cal R}_{G}^{(v)}$ that is not a fan, $v$ will have
spine-degree greater than 3, a contradiction to the first property of
Lemma~\ref{lem:atmost2nonfuns}.\\

\noindent{\em Case 2:} $v\in A_2\cup A_3$. If there exist a double rooted graph in 
${\cal R}_{G}^{(v)}$ that is not a fan, $v$ will 
have spine-degree greater than 2, therefore $v$ must be central, a contradiction.\\

b)  Let $v\in A_2$, and suppose that ${\cal R}_{G}^{(v)}$ can be contracted to two edges with 
$v$ as their unique common endpoint, 
or to a triangle. As $v$ belongs to a central block, $G$ can be contracted to a graph in 
${\cal O}_{5}$, a contradiction to the third property 
of Lemma~\ref{goinaldirs}.
\end{proof}

\subsection{Directional obstructions}

Let $G\in {\cal Q}$ and let $B_{0},B_{1},\ldots,B_{r},B_{r+1}$ be the spine blocks of $G$.
Notice first that from Lemma~\ref{basic_structure}.6 for every $i\in \{1,\ldots, r\}$, 
$|A_{2}\cap V(B_{i})|\leq 1$. Also, from Lemma~\ref{basic_structure}.6, 
if $A_{2}\cap V(B_{i})=\{v\}$, then $v$ is a light cut-vertex. 
For $i\in\{1,\ldots,r\}$, we define the rooted graphs ${\bf B}_{i}^{*}$ as follows: if 
$A_{2}\cap V(B_{i})=\{v\}$, then ${B}^{*}_{i}$ is the union of $B_{i}$ and 
the underlying graph of the unique rooted graph in ${\cal R}_{G}^{(v)}$ 
(this rooted graph is unique and its underlying graph is a hair block of $G$ from the second 
statement of Lemma~\ref{lem:interfdans}). If $A_{2}\cap V(B_{i})=\emptyset$, then 
${B}_i^{*}$ is $B_i$.
We finally define the rooted graph ${\bf B}^{*}_{i}=(B_{i}^{*},\{c_{i}\},\{c_{i+1}\})$ for $i\in\{1,\ldots,r\}$.

\begin{figure}[h]
\begin{center}
 \scalebox{0.85}{\includegraphics{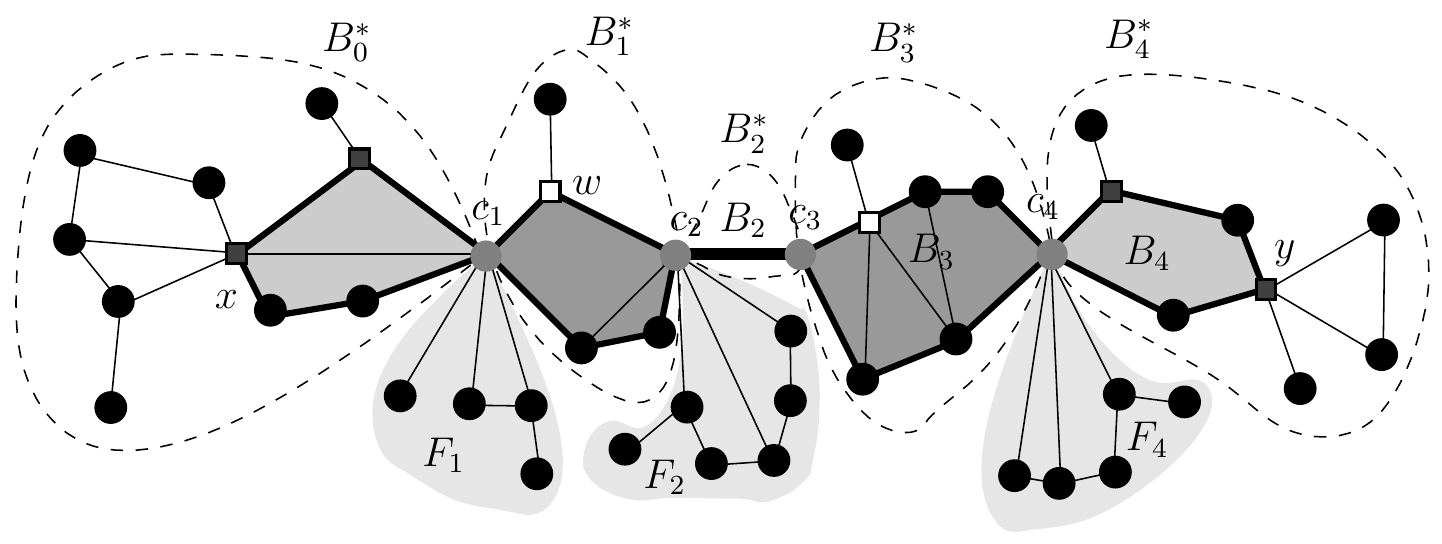}}\end{center}
\caption{A graph $G$, the  extended extremal blocks $B_{0}^*$ and $B_{4}^{*}$ the 
extended central  blocks $B_{1}^*,B_{2}^*,$ and $B_{3}^*$ and 
the rooted graphs $F_{1},F_{2}$ and  $F_{4}$ ($F_{3}$ is the graph consisting only of the 
vertex $c_{3}$, doubly rooted on $c_{3}$).}
\label{FanRoossstObs}
\end{figure}

We also define ${B}_{0}^{*}$ as follows: consider the unique graph $B_0$ in 
${\cal C}_{G}^{(c_1)}$ that, when doubly rooted on $c_{1}$, is not a 
fan, does not contain any edge of the central blocks of $G$ {and does not contain $c_{r+1}$}. 
Then ${\bf B}_{0}^{*}=(B_{0},\emptyset,\{c_{1}\})$. Analogously, we define  
$B_{r+1}^{*}$ by considering the graph $B_{r+1}$ of  the unique rooted graph  in 
${\cal R}_{G}^{(c_{r+1})}$ 
that, when doubly rooted on $c_{r+1}$, is not a fan, does not contain any edge of the 
central blocks of $G$ {and does not contain $c_{1}$}.
Then  ${\bf B}_{r+1}^{*}=(B_{r+1},\{c_{r+1}\},\emptyset)$.
Finally, we define for each $i\in\{1,\ldots,r+1\}$
the graph $F_{i}$ that is the union of all the graphs of the rooted graphs in 
${\cal R}_{G}^{(c_{i})}$ that are fans
(when performing the union, the vertex $c_{i}$ stays the same), and in the case where 
${\cal R}_{G}^{(c_{i})}$ is empty, then $F_{i}$ is the trivial graph $(\{c_i\},\emptyset)$.
We set ${\bf F}_{i}=(F_{i},\{c_i\},\{c_{i}\})$ $i\in\{1,\ldots,r+1\}$ and we  call the rooted graphs  
${\bf F}_{1},\ldots,{\bf F}_{r+1}$ {\em extended fans} of $G$.
We call ${\bf B}_{0}^{*},{\bf B}_{1}^{*},\ldots,{\bf B}_{r}^{*},{\bf B}^{*}_{r+1}$ the 
{\em extended  blocks} of the graph $G\in{\cal Q}$
and we naturally distinguish them in {\em central} and {\em extremal} ({\em left} or 
{\em right}), depending of type of the blocks that contain them.
Notice that 
$$\{E({ B}_{0}^{*}),E({ F}_{1}),E({ B}_{1}^{*}),E({ F}_2),\ldots,E({ F}_{r}),E({ B}_{r}^{*}),
E({ F}_{r+1}),E({ B}_{r+1}^{*})\}$$ is a partition of the edges of $G$.
 
\begin{figure}[h]
\begin{center}
\scalebox{0.6}{\includegraphics{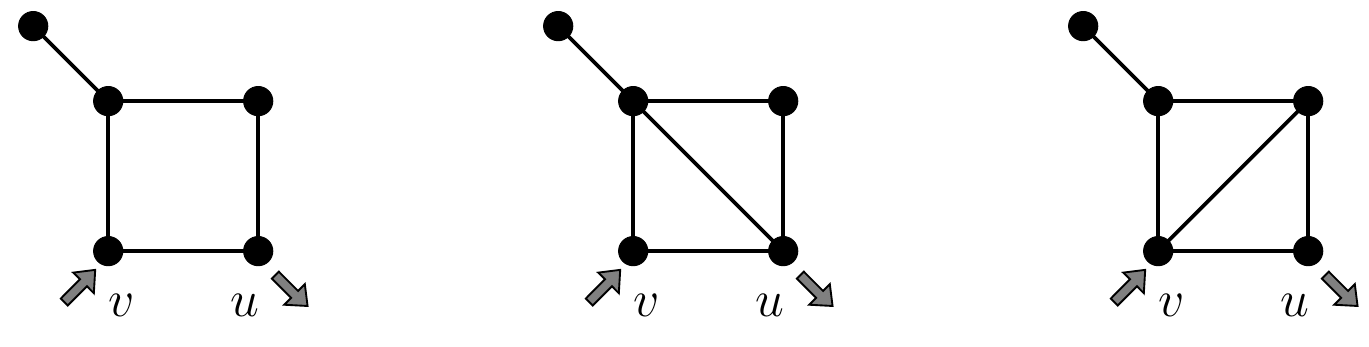}}
\end{center} 
\caption{The set of rooted graphs ${\cal L}$ containing three rooted graphs each of the form 
$(G,\{v\},\{u\})$.} 
\label{morein32sssminorc} 
\end{figure}

\begin{lemma}
\label{cnrptal}
Let $G\in {\cal Q}$ and let ${\bf B}_{1}^{*},\ldots,{\bf B}_{r}^{*}$ be the central  blocks of $G$. 
None of the rooted graphs in the set ${\cal L}$ of Figure~\ref{morein32sssminorc}
is a contraction of ${\bf B}_{i}^*$ if and only if  $\cmp({\bf B}_{i}^*)\leq 2$.
\end{lemma}

\begin{proof}  
Clearly, for every graph $H\in\cal L$, $\cmp(H,\{v\},\{u\})=3$, therefore if ${\bf B}_{i}^*$ can 
be contracted to a graph in $\cal L$, according to Lemma~\ref{lem:clos},  $\cmp({\bf B}_{i}^*)\geq 3$.

Let now ${\bf B}_{i}^*$ be an central extended block of $G$ that cannot be contracted to a graph in 
$\cal L$. 
If $B_{i}$ does not contain some light cut-vertex, we define 
${\bf G}_{B_{i}}=(B_{i},\{c_i\},\{c_{i+1}\})$. Notice that ${\bf B}_{i}^{*}={\bf G}_{B_{i}}$ and, as
from Lemma~\ref{bolek} $\cmp({\bf G}_{B_{i}^{}})\leq 2$, we are done.

In the remaining case, where $B_i$ contains a light cut-vertex, say $c$, observe that $c$  
cannot be adjacent via an outer edge to $c_i$, or else $B_{i}^{*}$ could be contracted to a graph in 
$\cal L$. Therefore, according to  Lemma~\ref{basic_structure}.7, $c$ is connected via an 
haploid edge with $c_{i+1}$. Notice that ${\bf G}_{B_i}=(B_i,\{c_i\},\{c_{i+1},c\})$.
According to Lemma~\ref{bolek}, $\cmp({\bf G}_{B_{i}})\leq 2$ and, according to 
Lemma~\ref{lem:interfdans}, ${\cal R}^{(c)}$ contains only a hair block, 
say $(H,\{c\},\{c\})$. Clearly $\cmp(H,\{c\},\{c\})=2$. Let 
${\bf G}_1=(G[\{c,c_{i+1}\}], \{c,c_{i+1}\}, \{c\})$, ${\bf G}_2=
(G[\{c, c_{i+1}\}], \{c\},\{c_{i+1}\}),$ and
${\bf G}={\bf glue}({\bf G}_{B_{i}}, {\bf G}_1, (H, \{c\},\{c\})$, ${\bf G}_2)$. From Lemma~\ref{glue}, 
$\cmp({\bf G})\leq 2$ and the lemma follows as ${\bf G}={\bf B}_{i}^{*}$.
\end{proof}

\begin{lemma}
\label{cnrptal2}
If ${\bf B}_{i}^{*}$ is one of the central extended  blocks of a graph $G\in{\cal Q}$, then either 
$\cmp({\bf B}_{i}^{*})\leq 2$ or $\cmp({\bf rev}({\bf B}_{i}^{*}))\leq 2$.
\end{lemma} 

\begin{proof} 
Proceeding towards a contradiction, from Lemma~\ref{cnrptal}, both ${\bf B}_{i}^{*}$ and 
${\bf rev}({\bf B}_{i}^{*})$ must contain one of the rooted graphs in 
Figure~\ref{morein32sssminorc} as a contraction. It is easy to verify that, in this case, either 
$B_{i}^{*}$
contains at least four cut-vertices, which contradicts to Lemma~\ref{basic_structure}.9 or 
$G$ can be contracted to either a graph in ${\cal O}_8$ (if the two roots are adjacent) 
or a graph in $O_{6}$ (if the two roots are not adjacent), a contradiction to 
Lemma~\ref{goinaldirs}.3.
\end{proof}

Let $G\in {\cal Q}$ and let ${\bf B}_{i}^{*}$ be one of the central extended blocks of $G$.

\begin{itemize}
\item If
${\bf rev}({\bf B}_{i}^{*})$ can be contracted to a graph in $\cal L$, then  we assign 
to ${\bf B}_{i}^*$ the label $\leftarrow$.

\item If 
${\bf B}_{i}^{*}$ can be contracted to a graph in $\cal L$, then  we assign 
to ${\bf B}_{i}^*$ the label $\rightarrow$.

\item If both ${\bf B}_{i}^{*}$ and ${\bf rev}({\bf B}_{i}^{*})$
can be contracted to a graph in $\cal L$, then  we assign 
to ${\bf B}_{i}^*$ the label $\leftrightarrow$.
\end{itemize}

\begin{figure}[h!] 
\begin{center} 
\scalebox{0.75}{\includegraphics{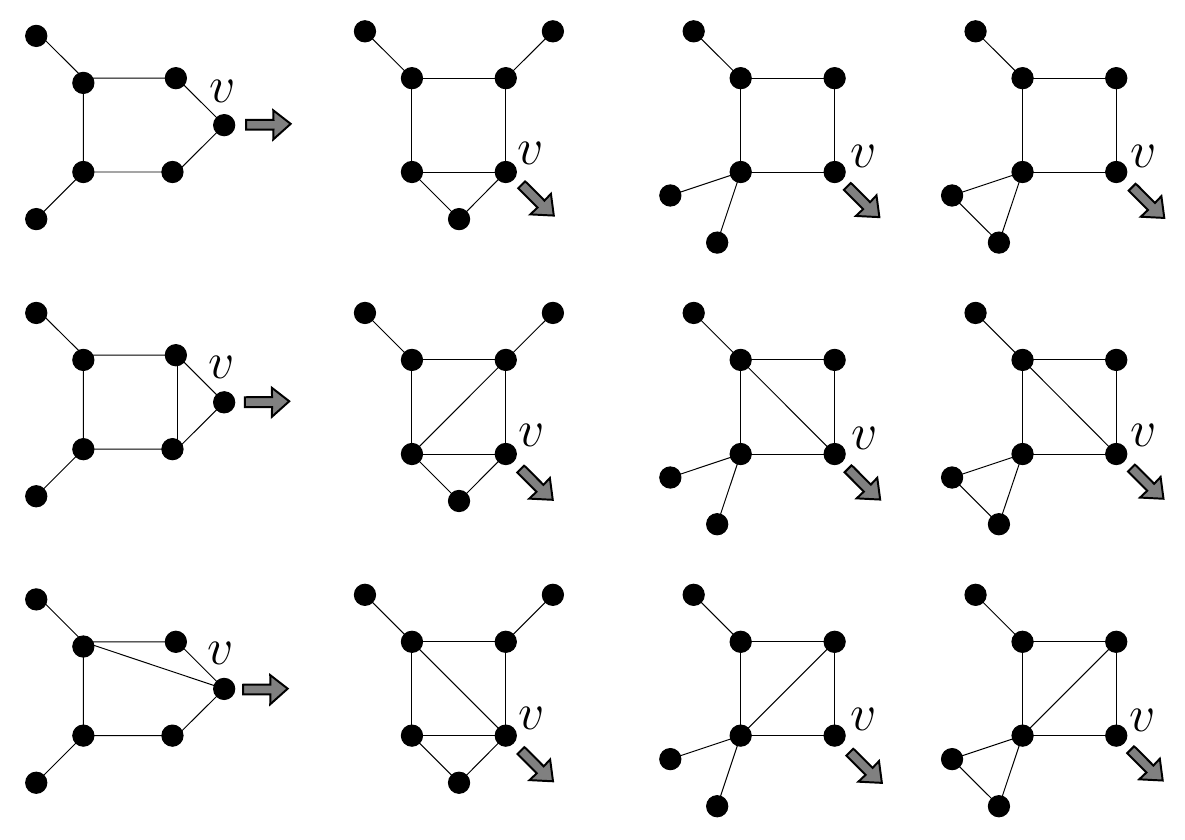}} 
\end{center}
\caption{The set of rooted graphs ${\cal B}$ containing 12 rooted graphs each of the form 
$(G,\emptyset,\{v\})$.} 
\label{direction} 
\end{figure}

\begin{lemma}
\label{ro4hk}
Let $G\in {\cal Q}$ and let ${\bf B}_0^*$ be the extended left extremal block of $G$.
None of the rooted graphs in the set ${\cal B}$ in Figure~\ref{direction}
is a contraction of ${\bf B}_{0}^{*}$ if and only if $\cmp({\bf B}_{0}^{*})\leq 2$.
\end{lemma}

\begin{proof} 
Clearly, for every graph $H\in\cal B$, $\cmp(H,\emptyset,\{u\})=3$. Therefore, if ${\bf B}_{0}^*$ can 
be contracted to a graph in $\cal B$, according to Lemma~\ref{lem:clos},  $\cmp({\bf B}_{0}^*)\geq 3$.

Suppose now  that ${\bf B}_{0}^*$ cannot be contracted to a graph in $\cal B$. 
We distinguish three cases according to the number of cut-vertices in $B_0$ 
(recall that, from Lemma~\ref{basic_structure}.4, $B_0$ can have at most 3 cut-vertices).\\

\noindent{\em Case 1:} $B_0$ contains only one cut-vertex, which  is $c_1$. Then,  
${\bf B}_{0}^*={\bf G}_{B_{0}}$ and the result follows  because of Lemma~\ref{bolek}.\\

\noindent{\em Case 2:} $B_0$ contains two cut-vertices, $c_1$ and $c$. If $B_{0}$ has 
not a chord or it has a chord and $c$ and $c_1$ are  incident to two different haploid faces of $B_0$ 
then we can assume that ${\bf G}_{B_{0}}=(B_{0},\{c\},\{c_1\})$ and 
from  Lemma~\ref{bolek}, $\cmp({\bf G}_{B_0})\leq 2$.
According to Lemma~\ref{lem:interfdans}.a, ${\cal R}^{(c)}$ is a fan, 
say $(F, \{c\},\{c\})$ and from Lemma~\ref{seasrchfun} $\cmp(F, \{c\},\{c\})\leq 2$. Let 
${\bf G}={\bf glue}((F, \{c\},\{c\}),{\bf G}_{B_0})$. From Lemma~\ref{glue},  $\cmp({\bf G})\leq 2$. 
Combining the fact that ${\bf G}=(B_{0}^{*},\{c\},\{c_1\})$ with  
Lemma~\ref{emptymake}, $\cmp({\bf B}_{0}^{*})\leq 2$.
In the remaining case  $c$ and $c_{1}$ are adjacent and $c$ is light.
Then ${\bf G}_{B_{0}}=({B}_{0},\emptyset,\{c,c_1\})$ and, from Lemma~\ref{bolek}, 
$\cmp({\bf G}_{B_{0}})\leq 2$. According to Lemma~\ref{lem:interfdans}.b, 
${\cal R}^{(c)}$ is a hair, say $(H, \{c\},\{c\})$. Let ${\bf G}_1=(G[\{c,c_{1}\}], \{c,c_{1}\}, \{c\})$, 
${\bf G}_2=(G[\{c, c_{1}\}], \{c\},\{c_{1}\})$ and
${\bf G}={\bf glue}({\bf G}_{B_{0}}, {\bf G}_1, (H, \{c\},\{c\})$, ${\bf G}_2)$. Notice that 
${\bf G}=(B_{0}^*,\emptyset,\{c_{1}\})={\bf B}_0^{*}$. From Lemma~\ref{glue}, we obtain that 
$\cmp({\bf G})\leq 2$, therefore $\cmp({\bf B}_{0}^{*})\leq 2$.\\

\noindent{\em Case 3:} $B_0$ contains three cut-vertices, $c_1$, $c$ and $x$.  
We first examine the case where there is a partition $\{X,Y\}$ of $\{c_1,c,x\}$ 
such that $|Y|=2$, $c_1\in Y$, the cut-vertex in $Y\setminus \{c_1\}$ is light, and $Y$ is an edge of 
$B_0$ that, in case $B_{0}$ is a chord, is haploid. In this case we claim that $\cmp({\bf B}_{0}^{*})$ 
$\leq 2$. Indeed,  we may assume that $c$ be the light cut-vertex of $Y\setminus\{c_1\}$, thus 
${\bf G}_{B_{0}}=(B_{0},\{x\},\{c,c_1\})$. According 
to Lemma~\ref{bolek}, $\cmp({\bf G}_{B_{0}})\leq 2$. From   Lemma~\ref{lem:interfdans}.a, 
${\cal R}^{(x)}$ 
is a fan, say $(F, \{x\},\{x\})$ and,  from   Lemma~\ref{lem:interfdans}.b, ${\cal R}^{(c)}$ 
contains only a hair block, say $(H,\{c\},\{c\})$. 
Let ${\bf G}_1=(G[\{c,c_{1}\}], \{c,c_{1}\}, \{c\})$, ${\bf G}_2=(G[\{c, c_{1}\}], \{c\},\{c_{1}\})$ and
${\bf G}={\bf glue}((F, \{x\},\{x\}), {\bf G}_{B_0}, {\bf G}_1, (H, \{c\},\{c\})$, ${\bf G}_2)$. Notice that 
${\bf G}=(B_{0}^*,\{x\},\{c_{1}\})$ and. From Lemma~\ref{glue}
$\cmp({\bf G})\leq 2$. Now, Lemma~\ref{emptymake} implies that $\cmp({\bf B}_{0}^{*})$ $\leq 2$ 
and the claim holds.

In the remaining cases, the following may happen:\medskip

1. None of $x$ and $c$ is adjacent to $c_1$ via an edge that, in case $B_{0}$  
has a chord, is haploid. In this case ${\bf B}_{0}^*$ can be contracted to the rooted 
graphs of the first column in Figure~\ref{direction}.\smallskip

2. Both $c_{1}$ and $x$ are light and only one of them, say $x$, is adjacent to $c_1$. 
In this case $B_{0}$ has a chord and 
either the edge $\{x,c_1\}$ is not haploid or  $\{x,c_1\}$ is haploid and belongs in the same 
haploid face with $c$. In the first case, ${\bf B}_{0}^*$ can be contracted to the second 
rooted graph of the second column in Figure~\ref{direction}
and in the second case   ${\bf B}_{0}^*$ can be contracted to the first and the third rooted 
graph of the second column in Figure~\ref{direction}.\smallskip

3. $c_{1}$ has only one, say $x$, heavy neighbour in $\{c,x\}$ such that,   in case $B_{0}$  
has a chord, the edge $\{c_{1},x\}$ is haploid.
In this case ${\bf B}_{0}^*$ can be contracted to the rooted graphs of the third and the fourth  
column in Figure~\ref{direction}.
\end{proof}

\begin{lemma}
\label{moreli}
Let $G\in {\cal Q}$ and let ${\bf B}_{r+1}^*$ be the extended right extremal block of $G$.
None of the rooted graphs in the set ${\cal C}$ in Figure~\ref{if9ton}
is a contraction of ${\bf B}_{r+1}^*$ if and only if  $\cmp({\bf B}_{r+1}^*)\leq 2$.
\end{lemma}

\begin{proof}
Clearly, for every graph $H\in\cal C$, $\cmp(H,\{u\}, \emptyset)=3$, therefore if ${\bf B}_{r+1}^*$ can 
be contracted to a graph in $\cal C$, according to Lemma~\ref{lem:clos},  
$\cmp({\bf B}_{r+1}^*)\geq 3$.

Suppose now  that ${\bf B}_{r+1}^*$ cannot be contracted to a graph in $\cal C$. 
We distinguish three cases according to the number of cut-vertices in $B_{r+1}$ 
(recall that, from Lemma~\ref{basic_structure}.4, $B_{r+1}$ can have at most 3 cut-vertices).\\

\noindent{\em Case 1:} $B_{r+1}$ contains only one cut-vertex, which of course is $c_1$. 
Notice that  ${\bf B}_{r+1}^*={\bf G}_{B_{r+1}}$ and the result follows  because of 
Lemma~\ref{bolek}.\\

\noindent{\em Case 2:}  If $B_{r+1}$ has not a chord or it has a chord and $c$ and $c_1$ 
are  incident to two different haploid faces of $B_{r+1}$ 
then we can assume that ${\bf G}_{B_{r+1}}=(B_{r+1},\{c_{r+1}\},\{c\})$ and 
from  Lemma~\ref{bolek}, $\cmp({\bf G}_{B_{r+1}})\leq 2$.
In any other case, $c_{r+1}$ and $c$ is on the boundary  in the same   haploid face of $B_{r+1}$  
and none of them belongs in the boundary of  the other.  Then ${\bf B}_{r+1}^{*}$ can be contracted to 
some of the rooted graphs in the second column of Figure~\ref{if9ton}.\medskip

\noindent{\em Case 3:} $B_{r+1}$ contains three cut-vertices, $c_{r+1}$, $c$ and $x$. 
If $c$ and $x$ are not adjacent, then  ${\bf B}_{r+1}^{*}$ can be contracted to 
some of the rooted graphs in the first column of Figure~\ref{if9ton}.
Otherwise, one, say $c$, of them will be light and the edge $\{x,c\}$
should be haploid. Then we can assume that ${\bf G}_{B_{r+1}}=(B_{r+1},\{c_{r+1}\},\{x,c\})$ 
and according to Lemma~\ref{bolek}, 
$\cmp({\bf G}_{B_{r+1}})\leq 2$.
From Lemma~\ref{lem:interfdans}, 
${\cal R}^{(x)}$ is a fan, say $(F, \{x\},\{x\})$ and ${\cal R}^{(c)}$ contains only a hair block, say 
$(H,\{c\},\{c\})$. 
Let ${\bf G}_1=(G[\{c,x\}], \{c,x\}, \{c\})$, ${\bf G}_2=(G[\{c, x\}], \{c\},\{x\})$ and
${\bf G}={\bf glue}({\bf G}_{B_{r+1}}, {\bf G}_1, $ $(H, \{c\},\{c\}), {\bf G}_2, (F, \{x\},\{x\}))$. Notice that 
${\bf G}=(B_{r+1}^*,\{c_{r+1}\},\{x\})$ and, because of Lemma~\ref{glue},
$\cmp({\bf G})\leq 2$. Applying Lemma~\ref{emptymake}, we conclude that 
$\cmp({\bf B}_{r+1}^{*})\leq 2$.
\end{proof}

\begin{figure}[h!] 
\begin{center} 
\scalebox{0.75}{\includegraphics{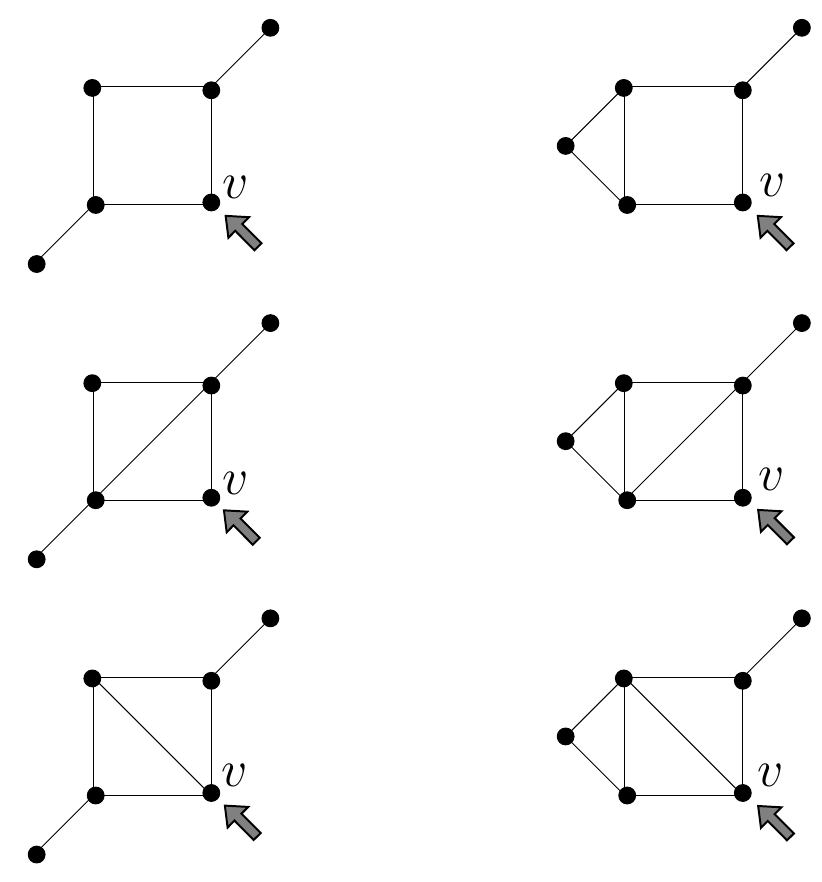}} 
\end{center}
\caption{The set of rooted graphs ${\cal C}$ containing six rooted graphs each of the form 
$(G,\{v\},\emptyset)$.} 
\label{if9ton} 
\end{figure}

\begin{lemma}
\label{tboths}
Let $G\in {\cal Q}$ and let ${\bf B}^*_{0}$ and ${\bf B}^{*}_{r+1}$  be the two extremal 
extended blocks of $G$.
 It is never the case that
 ${\bf B}^*_{0}$ contains some graph in ${\cal B}$ and 
  $\rrev({\bf B}^*_{0})$ contains some rooted graph in ${\cal C}$.
  Also it is never the case that
 ${\bf B}^{*}_{r+1}$ contains some graph in ${\cal C}$ and 
  $\rrev({\bf B}^{*}_{r+1})$ contains some rooted graph in ${\cal B}$.
\end{lemma}

\begin{proof} 
Let ${\bf G}$ be a rooted graph in $K=\{\rrev({\bf B}^*_{0}),{\bf B}^{*}_{r+1}\}$.
We distinguish the following cases, that apply for both rooted graphs in $K$:\\

\noindent{\em Case 1:} $\bf G$ can be contracted to a graph in the first column of Figure~\ref{if9ton}
 and $\rrev(\bf G)$ to a graph in the first column of Figure~\ref{direction}. Notice that every cut-vertex 
 of $\bf G$ cannot be connected with an outer edge and therefore  $\bf G$ can be contracted to graph 
 in ${\cal O}_6$.\\

\noindent{\em Case 2:} $\bf G$ can be contracted to a graph in the first column of Figure~\ref{if9ton}
 and $\rrev(\bf G)$ to a graph in the second column of Figure~\ref{direction}. 
 Notice that the two cut-vertices
of $\bf G$, that are other than the central cut-vertex, cannot be connected with an outer edge 
and therefore  
$\bf G$ can be contracted to graph in ${\cal O}_7$.\\

\noindent{\em Case 3:} $\bf G$ can be contracted to a graph in the first column of Figure~\ref{if9ton}
 and $\rrev(\bf G)$ to a graph in the third or fourth column of Figure~\ref{direction}. Notice that the 
 two cut-vertices
of $\bf G$, that are other than the central cut-vertex, cannot be connected with an outer edge 
and therefore  
$\bf G$ can be contracted to graph in ${\cal O}_8$.\\

\noindent{\em Case 4:} $\bf G$ can be contracted to a graph in the second column of 
Figure~\ref{if9ton}
 and $\rrev(\bf G)$ to a graph in the first column of Figure~\ref{direction}. Notice that the two 
 cut-vertices
of $\bf G$, that are other than the central cut-vertex, cannot be connected with an outer edge 
and therefore  
$\bf G$ can be contracted to graph in ${\cal O}_7$.\\

\noindent{\em Case 5:} $\bf G$ can be contracted to a graph in the second column of 
Figure~\ref{if9ton}
 and $\rrev(\bf G)$ to a graph in the second column of Figure~\ref{direction}. Notice that there must be 
an haploid face containing only the central cut-vertex, therefore $\bf G$ can be contracted either to
the graph in ${\cal O}_2$  or to a graph in ${\cal O}_4$ (depending whether $\bf G$ can be 
contracted to the last graph in the second column of Figure~\ref{if9ton} or not) .\\
 
\noindent{\em Case 6:} $\bf G$ can be contracted to a graph in the second column of 
Figure~\ref{if9ton}
 and $\rrev(\bf G)$ to a graph in the third or fourth column of Figure~\ref{direction}. 
 Notice that the light cut-vertex of ${\bf  G}$ can not be connected via an haploid edge 
 with the central cut-vertex, therefore $\bf G$ can be contracted to a graph in ${\cal O}_7$ 
 either to a graph in ${\cal O}_9$ (depending whether the central cut-vertex is connected 
 via haploid edge with a heavy cut-vertex or not).
\end{proof}

\begin{itemize}
\item If ${\bf B}^*_{0}$ contains some graph in ${\cal B}$  then we assign 
to ${\bf B}^*_{0}$ the label $\leftarrow$.

\item If $\rrev({\bf B}^*_{0})$ contains some graph in ${\cal C}$ then we assign 
to ${\bf B}^*_{0}$ the label $\rightarrow$.

\item If ${\bf B}^{*}_{r+1}$ contains some graph in ${\cal C}$  then we assign 
to ${\bf B}^{*}_{r+1}$ the label $\leftarrow$.

\item If  $\rrev({\bf B}^{*}_{r+1})$ contains some graph in ${\cal B}$ then we assign 
to ${\bf B}^{*}_{r+1}$ the label $\rightarrow$.

\item If neither ${\bf B}^*_{0}$ contains some graph in ${\cal B}$
nor  $\rrev({\bf B}^*_{0})$ contains some graph in ${\cal C}$ then we assign 
to ${\bf B}^*_{0}$ the label $\leftrightarrow$.

\item If neither ${\bf B}^{*}_{r+1}$ contains some graph in ${\cal C}$ 
nor    $\rrev({\bf B}^{*}_{r+1})$ contains some graph in ${\cal B}$ then we assign 
to ${\bf B}^{*}_{r+1}$ the label $\leftrightarrow$.
\end{itemize}

\begin{lemma}
\label{tplerr}
Let $G\in {\cal Q}$ and let ${\bf B}^{*}_0,{\bf B}_{1}^{*},\ldots,{\bf B}_{r}^{*},{\bf B}^{*}_{r+1}$ 
be the extended  blocks of $G$.
It is not possible that one of these extended blocks is labeled with $\leftarrow$ and an other 
with $\rightarrow$.
\end{lemma}

\begin{proof} 
We distinguish two cases according to the labelling of ${\bf B}_{0}^{*}$:\\
 
\noindent{\em Case 1:} Suppose that ${\bf B}_{0}^{*}$ is labeled $\leftarrow$ 
and that ${\bf B}_{i}^{*}$, for some $i\in\{1,\ldots,r+1\}$, is labeled $\rightarrow$. According to their 
respective labels, $(B_{0}^{*},\emptyset,\{c_1\})$
can be contracted to a graph in $\cal B$ and,  if $i\leq r$,
 ${\bf rev}({\bf B}_{i}^{*})=(B_{i}^{*}, \{c_{i+1}\},\{c_{i}\})$ can be contracted to a graph in 
${\cal L}$, otherwise ${\bf rev}({\bf B}_{r+1}^{*})=(B_{r+1}^{*},\emptyset,\{c_{r+1}\})$ can 
be contracted to a rooted graph in ${\cal B}$.
Notice that if  $i\leq r$, $G[V(G)\setminus (V(B_{1}^{*})\cup\cdots\cup 
V(B_{i-1}^{*}))],\emptyset,\{c_i\}\}$ can be contracted to a graph in the third and forth columns of 
$\cal B$. By further contracting all edges of $E(B_{1}^{*})\cup \cdots\cup E(B_{i-1}^{*})$ we obtain 
a graph in ${\cal O}_{11}$, a contradiction.\\

\noindent{\em Case 2:} Suppose now that ${\bf B}_{0}^{*}$ is labeled $\rightarrow$ 
and that ${\bf B}_{i}^{*}$, for some $i\in\{1,\ldots,r+1\}$, is labeled $\leftarrow$. According to their 
respective labels, $\rrev({\bf B}_{0})^{*}=(B_{0}^{*},\{c_1\},\emptyset)$
can be contracted to a graph in $\cal C$ and,  if $i\leq r$, 
$(B_{i}^{*}, \{c_{i}\},\{c_{i+1}\})$ can be contracted to a graph in 
$\cal L$, otherwise $(B_{r+1}^{*}, \{c_{r+1}\},\emptyset)$ can be contracted to a graph in
${\cal C}$. 
Notice that if  $i\leq r$, $G[V(G)\setminus (V(B_{1}^{*})\cup\cdots\cup 
V(B_{i-1}^{*}))],\{c_i\},\emptyset\}$ can be be contracted to a graph in the first column of $\cal C$. 
By further contracting all edges of $E(B_{1}^{*})\cup \cdots\cup E(B_{i-1}^{*})$ we obtain 
a graph in ${\cal O}_{12}$, a contradiction.
\end{proof}

\subsection{Putting things together}
\begin{lemma}
\label{tll0o}
${\cal Q}=\emptyset$.
\end{lemma}

\begin{proof} 
Suppose in contrary that ${\cal Q}$ contains some graph $G$. Let 
${\bf B}_{0}^{*},{\bf B}_{1}^{*},\ldots,{\bf B}_{r}^{*},$ ${\bf B}^{*}_{r+1}$ 
and  ${\bf F}_{1},\ldots,{\bf F}_{r+1}$ be the extended blocks and fans of $G$, respectively.
From Lemma~\ref{tplerr}, we can assume that the extended blocks of $G$ 
are all labeled either $\rightarrow$ or $\leftrightarrow$ (if this is not the case, 
just reverse the ordering of the blocks).
By the labelling of ${\bf B}_{0}^{*}$,
none of the rooted graphs in the set ${\cal B}$ is a contraction of ${\bf B}_{0}^{*}$ therefore, 
from Lemma~\ref{ro4hk}, 
$\cmp({\bf B}_{0}^{*})\leq 2$. Also as none of the rooted graphs ${\bf B}_{i}^{*}$, $i=1,\ldots,r,$ 
can be contracted to a graph 
in ${\cal L}$, from Lemma~\ref{cnrptal}, it follows that $\cmp({\bf B}_{i}^{*})\leq 2$. 
We distinguish two cases according to the labelling of ${\bf B}_{r+1}^{*}$.
 If the labelling is $\leftrightarrow$, then ${\bf B}_{r+1}^{*}$ cannot be contracted to a graph 
 in ${\cal C}$. If the labelling is $\rightarrow$, then $\rrev({\bf B}_{r+1}^{*})$ can be contracted to
  a graph in ${\cal B}$ and, according to Lemma~\ref{tplerr}, 
${\bf B}_{r+1}^{*}$ cannot be contracted to a graph in ${\cal C}$. Thus, in both cases, from 
Lemma~\ref{moreli},  
$\cmp({\bf B}_{r+1}^{*})\leq 2$. Notice that 
$(G,\emptyset,\emptyset)={\bf glue}({\bf B}_{0}^{*},{\bf F}_{1},{\bf B}_{1}^{*},\ldots,
{\bf F}_{r},{\bf B}_{r}^{*},{\bf F}_{r+1},{\bf B}_{r+1}^{*})$ and, from Lemma~\ref{glue}, 
$\cmp(G,\emptyset,\emptyset)\leq 2$.
This implies that $\cmp(G)\leq 2$, a contradiction to the first property of Lemma~\ref{goinaldirs}.
\end{proof}

\begin{corollary}
\label{cor:cs-obs}
$\obs_{\preceq}({\cal G}[\cms,2])=\obs_{\preceq}({\cal G}[\cs,2])$.
\end{corollary}

\begin{proof}
It is easy to check that for every $H\in \obs_{\preceq}({\cal G}[\cms,2])$:
\begin{enumerate}
\item $\cs(H)\geq 3$,
\item for every proper contraction $H'$ of $H$ it holds that $\cs(H')\leq 2$, 
\end{enumerate}
therefore $\obs_{\preceq}({\cal G}[\cms,2])\subseteq \obs_{\preceq}({\cal G}[\cs,2])$.

If there exist a graph $H\in \obs_{\preceq}({\cal G}[\cs,2])\setminus\obs_{\preceq}({\cal G}[\cms,2])$,
then $\cs(H)$ $\geq 3$. Notice that the connected search number of a graph is always bounded 
from the
monotone and connected search number, {as a complete monotone and connected search strategy 
is obviously a complete connected search strategy, therefore} $\cms(H)\geq 3$, which means 
that there 
exist a graph $H'\in \obs_{\preceq}({\cal G}[\cms,2])$ such that $H'\preceq H$. Furthermore, 
since $H'$ is 
a proper contraction of $H$, according to Lemma~\ref{lem:cs-clos}, $\cs(H)\geq\cs(H')$. 
As we have already stressed that $\cs(H')\geq 3$ we reach a 
contradiction to the minimality (with respect of $\preceq$) of $H$.
\end{proof}

\section{General Obstructions for $\cms$} 

As we mentioned before, for $k>2$, we have no guarantee that the set 
$\obs_{\preceq}({\cal G}[\ cmms,k])$
 is a finite set. In this section we  prove that when this set is finite its size should be 
 double exponential in $k$. Therefore, it seems hard to
extend our results for $k\geq 3$ as, even if we somehow manage to prove that the 
obstruction set for a specific $k$ is finite, then this set would contain more than 
$2^{2^{\Omega(k)}}$ graphs.

We will describe a procedure that generates, for each $k$, a set of at least 
$\frac{4}{3}(\frac{5}
{2})^{3\cdot 2^{k-2}}$ non-isomorphic graphs that have connected and monotone search number 
$k+1$ and are  contraction-minimal
with respect to this property. Hence, these $\frac{4}{3}(\frac{5}{2})^{3\cdot 2^{k-2}}$ graphs will belong 
to $\obs_{\preceq}({\cal G}[\cms,k])$.

We define for every $k\geq1$  a set of rooted graphs, namely the set of  
{\em obstruction-branches} denoted ${\bf Br}(k)$, as follows:\medskip

\noindent {\bf For $k=1$:} The set ${\bf Br}(1)$ consists of the five graphs of 
Figure~\ref{FanRootObs} rooted at $v$.\medskip

\noindent {\bf For $k=l>1$:} The set ${\bf Br}(l)$ is constructed by choosing two {\em branches} 
of the set ${\bf Br}(l-1)$ and identify the two roots to a single vertex, say $v$. Then we add a 
new edge with $v$ as an endpoint, say $\{u,v\}$, and we root this branch to $u$. We will refer 
to this edge as the {\em trunk} of the branch.\medskip

Let $f(k)$ be the number of branches of ${\bf Br}(k)$. Notice that $f(1)=5$ and $f(k)$ is equal 
to the number of ways we can pick two branches of ${\bf Br}(k-1)$, with repetition. Therefore:

\begin{eqnarray}
f(k)&=&{f(k-1)+2-1 \choose 2}=  {f(k-1)+1 \choose 2}\nonumber \\
&=&\frac{f(k-1)^2+f(k-1)}{2}\geq \frac{f(k-1)^2}{2}\nonumber \\
&\geq&\frac{\big(\frac{f(k-2)^2}{2}\big)^2}{2}=\frac{f(k-2)^{2^2}}{2^{2+1}}\geq\frac{f(k-3)^{2^3}}
{2^{2^2+2+1}}\nonumber \\
&\geq&\cdots\geq \frac{f(1)^{2^{k-1}}}{2^{2^{k-2}+\cdots+2+1}}=\frac{5^{2^{k-1}}}{2^{2^{k-1}}}
=2\Big(\frac{5}{2}\Big)^{2^{k-1}}\nonumber
\end{eqnarray}

Let ${\cal O}_{\bf Br}(k)$ be the set containing the graphs obtained by choosing three rooted 
branches
of ${\bf Br}(k)$, with repetitions, and identify the three roots. {Notice that any two such 
 selections produce two non-isomorphic graphs.} We are going to prove that 
${\cal O}_{\bf Br}(k)\subseteq 
\obs_{\preceq}({\cal G}[\cms,k+1])$. Notice that:

\begin{eqnarray}
|{\cal O}_{\bf Br}(k)|&=&{f(k)+3-1 \choose 3}={f(k)+2 \choose 3}\nonumber \\
&=&\frac{(f(k)+2)(f(k)+1)f(k)}{6}\nonumber\\
&=&\frac{f(k)^3+3f(k)^2+2f(k)}{6}\geq \frac{f(k)^3}{6}\nonumber\\
&\geq&\frac{4}{3}\Big(\frac{5}{2}\Big)^{3\cdot2^{k-1}}\nonumber
\end{eqnarray}

\begin{figure}
\begin{center} 
\scalebox{0.8}{\includegraphics{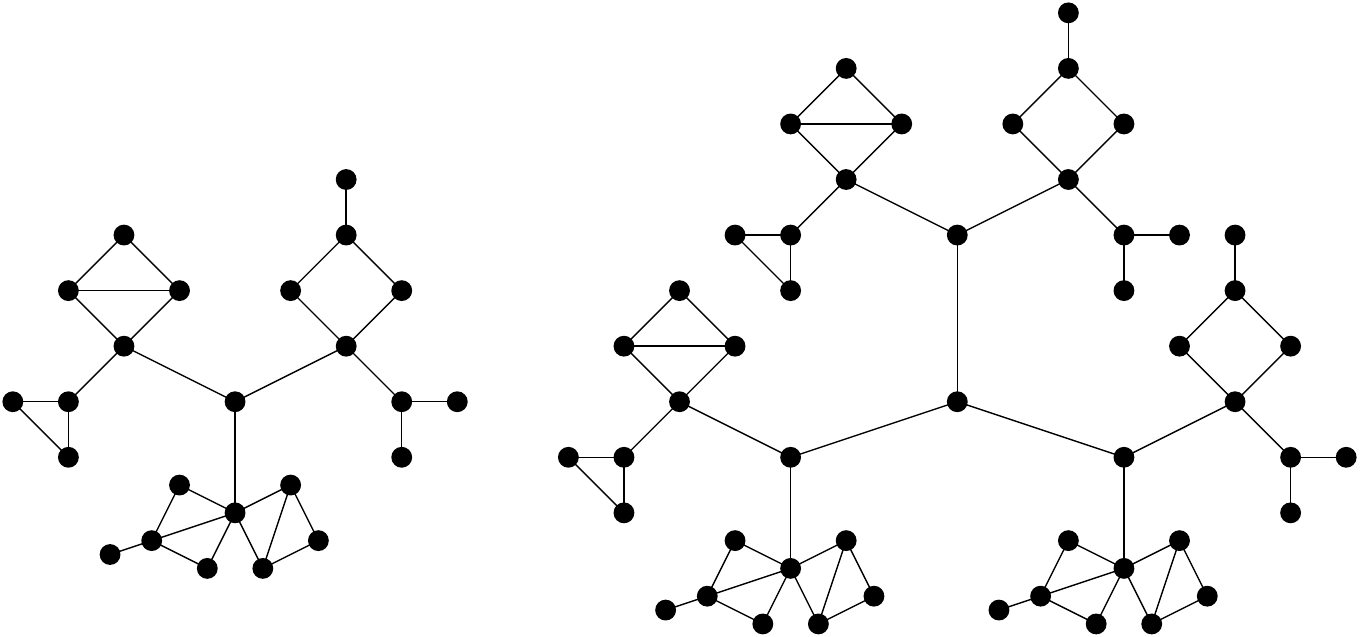}}
\end{center}
\caption{The left graph belongs to ${\cal O}_{\bf Br}(2)$ and the right to ${\cal O}_{\bf Br}(3)$.} 
\label{brobstr} 
\end{figure}

Hence the cardinality of $\obs_{\preceq}({\cal G}[\cms,k])$ is at least 
$\frac{4}{3}\Big(\frac{5}{2}\Big)^{3\cdot2^{k-2}}$. In order to prove this we need the 
following Lemmata.

\begin{lemma}\label{boundLem0}
Let $B\in {\bf Br}(k)$ and let $v$ be its root. There does not exist a complete 
monotone and connected search strategy for $B$ 
that uses $k$ searchers, such that  the first edge being cleaned is the trunk of $B$.
\end{lemma}

\begin{proof}
We are going to prove this by induction. We can easily check that for $k=1$ the claim holds. 
Let $B\in {\bf Br}(k)$ and let $v$ be its root and $u$ the other endpoint of the trunk. 
Since we are forced to clean $B$, in a connected and monotone manner, 
with a search strategy, say $\cal S$, that first cleans $\{u,v\}$, a searcher must be 
placed in $u$ during each step of $\cal S$, therefore we must clean a $(k-1)$-level 
branch using $k-1$ searchers that first clean the trunk of this branch, 
which contradicts the induction hypothesis.  
\end{proof}

\begin{corollary}\label{boundCor1}
Let $G\in {\cal O}_{\bf Br}(k)$, then $\cms(G)>k+1$. 
\end{corollary}

\begin{proof}
Let $G\in {\cal O}_{\bf Br}(k)$. Notice that $G$ consists of three $k$-level obstruction branches, 
say $B_1$, $B_2$ and $B_3$. If there exist a complete monotone and connected search strategy 
$\cal S$ that uses $k+1$ searchers, then from Lemma~\ref{boundLem0} $\cal S$ cannot start 
by placing searchers in the central vertex, i.e. the vertex where $B_1$, $B_2$ and $B_3$ 
are connected. Therefore, $\cal S$ starts by placing searchers in a vertex of $B_1$, $B_2$ or 
$B_3$ and consequently the first edge cleaned belongs to this branch. Notice that the first 
time that a searcher is placed on the central vertex the connectivity and monotonicity of 
$\cal S$ force us to clean a $k$-level branch with $k$ searchers, which is impossible 
according to Lemma~\ref{boundLem0}. 
\end{proof}

\begin{lemma}\label{boundLem2}
Let $B\in {\bf Br}(k)$ and let $v$ be its root. 
\begin{itemize}
\item[a)] There exist a complete monotone and connected search strategy for $B$ that uses 
$k+2$ searchers, such that in each step a searcher occupies $v$. 
\item[b)] There exists a complete monotone and connected search strategy for $B$ 
that uses $k+1$ searchers, such that  the first edge being cleaned is the trunk of $B$.
\item[c)] There exist a complete monotone and connected search strategy for $B$ that uses 
$k+1$ searchers, such that the last edge being cleaned is the trunk of $B$. 
\end{itemize}
\end{lemma}

\begin{proof}
a) We are going to prove this by induction. We can easily check that for $k=1$ the claim holds. 
Let $B\in {\bf Br}(k)$ and let $v$ be its root and $u$ the other endpoint of the trunk. 
We are going to describe a search strategy $\cal S$ with the properties needed. 
We place a searcher in $v$ and  as second searcher in $u$. According to the induction 
hypothesis for each one of the two $(k-1)$-level branches connected to $u$ there exists a 
complete monotone and connected search strategy that uses $k+1$ searchers such that in 
each step a searcher occupies $u$, therefore we can continue by cleaning one of these 
$(k-1)$-level branches and then clean the other.\medskip

b) We are going to prove this by induction. For $k=1$ the claim is trivial. Let $B\in {\bf Br}(k)$ 
and let $v$ be its root and $u$ the other endpoint of the trunk. There are two $(k-1)$-level 
branches connected to $u$, say $B_1$ and $B_2$. The search strategy, say $\cal S$, 
with the properties needed is the following: we place a searcher in $v$ and then slide him to $u$. 
According to the first claim of Lemma~\ref{boundLem2} there exists a complete monotone and 
connected search strategy ${\cal S}_1$ for $B_1$ that uses $k+1$ searchers such that in each step a 
searcher occupies $u$. By the induction hypothesis there exists a complete monotone and connected 
search strategy ${\cal S}_2$ for $B_2$ that uses $k$ searchers such that the first edge cleaned is the 
trunk of $B_2$. Using these two search strategies we can start by cleaning $B_1$, keeping in all 
times a searcher in $u$, and then we can clean $B_2$.\medskip

c) We are going to prove this by induction. Notice that for $k=1$ the claim holds. Let 
$B\in {\bf Br}(k)$ and let $v$ be its root and $u$ the other endpoint of the trunk. There are two 
$(k-1)$-level branches connected to $u$, say $B_1$ and $B_2$. According to the 
induction hypothesis there exist a complete monotone and connected search strategy ${\cal S}_1$ for 
$B_1$ that uses $k$ searchers such that the last edge cleaned is the trunk of $B_1$. Moreover, 
according to the first claim of Lemma~\ref{boundLem2} there exists a complete monotone and 
connected search strategy ${\cal S}_2$ for $B_2$ that uses $k+1$ searchers such that in each step a 
searcher occupies $u$. Using these two search strategies  we can clean $B$, in a monotone and 
connected manner, as follows: we start by cleaning $B_1$ then we  clean $B_2$, keeping in all times 
a searcher in $u$, and then we clean $\{u,v\}$.
\end{proof}

\begin{lemma}\label{boundLem3}
Let $G\in {\cal O}_{\bf Br}(k)$ and $B\in {\bf Br}(k)$ one of the three branches of $G$. 
If we contract an edge of $B$ there exist a complete monotone and connected 
search strategy for $B$ that uses 
$k+1$ searchers, such that in each step a searcher occupies $v$. 
\end{lemma}

\begin{proof}
We are going to prove this by induction. It is easy to check that for $k=1$ the claim is true.
Let $v$ be the root of $B$ and $u$ the other endpoint of the trunk, $B_1$ and $B_2$ the two 
$(k-1)$-level branches connected to $u$ and $e\in E(B)$ the edge contracted. 
We distinguish to cases:\medskip\\
\noindent{\em Case 1:} $e\in E(B_1)\cup E(B_2)$. We can assume that $e$ is an edge of $B_1$. 
We are going to describe a search strategy $\cal S$ for $B$ with the properties needed. 
We place a searcher in $v$ and  a second searcher in $u$. From the induction hypothesis 
there exists a complete monotone and connected search strategy ${\cal S}_1$ for $B_1$ 
that uses $k$ searchers, such that in each step a searcher occupies $u$. Moreover, 
according to the second claim of Lemma~\ref{boundLem2} there exists a complete 
monotone and connected search strategy ${\cal S}_2$ for $B_2$ that uses $k$ searchers 
such that the first edge cleaned is the trunk of $B_2$.  Using these two search strategies 
we can start by cleaning $B_1$, keeping in all times a searcher in $u$, and then we can 
clean $B_2$. Notice that this search strategy uses $k+1$ searchers and during each step 
a searcher occupies $v$.
\medskip

\noindent{\em Case 2:} $e=\{u,v\}$. According to the first property of Lemma~\ref{boundLem2}, 
for each one of $B_1$ and $B_2$ there exists a complete monotone and connected search 
strategy that uses $k+1$ searchers such that in each step a searcher occupies $v$. 
Hence we can clean $B$ starting by cleaning $B_1$, keeping in all times a searcher in 
$v$, and then clean $B_2$.
\end{proof}

\begin{corollary}\label{boundCor2}
If $G\in {\cal O}_{\bf Br}(k)$ and $G'$ be a contraction of $G$, then $\cms(G')=k+1$. 
\end{corollary}

\begin{proof}
It suffices to prove this claim for a single edge contraction.
Let $G\in {\cal O}_{\bf Br}(k)$, let $B_1$, $B_2$, and $B_3$ be the three $k$-level obstruction-
branches of $G$ connected to $v$, $e\in E(G)$ the edge contracted and $G'$ the graph obtained 
from $G$ after the contraction of $e$. We can assume that $e\in E(B_2)$. We are going to describe a 
complete monotone and connected search strategy $\cal S$ for $G$. From the third claim of Lemma~
\ref{boundLem2} we know that there exist a complete monotone and connected search strategy 
${\cal S}_1$ for $B_1$ that uses $k+1$ searchers, such that the last edge cleaned is the trunk of 
$B_1$. From Lemma~\ref{boundLem3} we know that there exist a complete monotone and connected 
search strategy ${\cal S}_2$ for $B_2$ that uses $k+1$ searchers, such that in each step a searcher 
occupies the root of $B_2$, in other words $v$. From the second claim of Lemma~\ref{boundLem2} 
we know that there exist a complete monotone and connected search strategy ${\cal S}_3$ for $B_3$ 
that uses $k+1$ searchers, such that the first edge cleaned is the trunk of $B_3$. Therefore, we can 
clean $G'$ starting by cleaning $B_1$ according to ${\cal S}_1$ (notice that the trunk of $B_1$ will be 
the last edge of $E(B_1)$ being cleaned), then clean $B_2$ according to ${\cal S}_2$, keeping in all 
times a searcher in $v$, and finish by cleaning $B_3$ according to ${\cal S}_3$.
\end{proof}

Combining Corollaries~\ref{boundCor1} and~\ref{boundCor2} we conclude that every graph in 
${\cal O}_{\bf Br}(k)$ is a contraction obstruction for the graph class ${\cal G}[\cms,k+1]$ 
and therefore ${\cal O}_{\bf Br}(k)\subseteq \obs_{\preceq}({\cal G}[\cms,k+1])$.

\nocite{*} 
￼\bibliographystyle{elsarticle-num}

\end{document}